\documentclass{amsart}
\usepackage[utf8]{inputenc}
\usepackage{geometry}
\usepackage{enumerate}
\usepackage{amsmath}
\usepackage{cite}
\usepackage{amsfonts}
\usepackage{mathrsfs}
\usepackage{galois}
\usepackage[toc,page]{appendix}
\usepackage{xcolor}
\usepackage[colorlinks, linkcolor = blue, anchorcolor = blue, citecolor = blue]{hyperref}

\newtheorem{Def}{Definition}[section]
\newtheorem{Thm}[Def]{Theorem}

\newtheorem{Lem}[Def]{Lemma}
\newtheorem{Prop}[Def]{Proposition}
\newtheorem{Cor}[Def]{Corollary}
\newtheorem{Rem}[Def]{Remark}

\numberwithin{equation}{section}

\newcommand{\AAa}{\mathbb{A}}
\newcommand{\BB}{\mathbb{B}}

\newcommand{\NN}{\mathbb{N}}

\newcommand{\RR}{\mathbb{R}}
\newcommand{\SSp}{\mathbb{S}}

\newcommand{\ZZ}{\mathbb{Z}}

\newcommand{\cA}{\mathcal{A}}

\newcommand{\cC}{\mathcal{C}}
\newcommand{\cD}{\mathcal{D}}
\newcommand{\cE}{\mathcal{E}}
\newcommand{\cF}{\mathcal{F}}

\newcommand{\cI}{\mathcal{I}}

\newcommand{\cL}{\mathcal{L}}
\newcommand{\cM}{\mathcal{M}}
\newcommand{\cN}{\mathcal{N}}

\newcommand{\cP}{\mathcal{P}}

\newcommand{\cR}{\mathcal{R}}
\newcommand{\cS}{\mathcal{S}}

\newcommand{\cV}{\mathcal{V}}

\newcommand{\cZ}{\mathcal{Z}}

\newcommand{\scC}{\mathscr{C}}

\newcommand{\scH}{\mathscr{H}}

\newcommand{\scM}{\mathscr{M}}

\newcommand{\scT}{\mathscr{T}}

\newcommand{\scX}{\mathscr{X}}

\newcommand{\bfF}{\textbf{F}}

\newcommand{\mbfE}{\mathbf{E}}
\newcommand{\mbfF}{\mathbf{F}}

\newcommand{\mbfI}{\mathbf{I}}

\newcommand{\mbfK}{\mathbf{K}}

\newcommand{\mbfM}{\mathbf{M}}

\newcommand{\Int}{\text{Int}}
\newcommand{\Clos}{\text{Clos}}

\newcommand{\spt}{\text{spt}}

\newcommand{\Sing}{\mathrm{Sing}}
\newcommand{\Reg}{\mathrm{Reg}}

\title{Mean Convex Smoothing of Mean Convex Cones}
\author{Zhihan Wang}
\address{Department of Mathematics, Princeton University, Fine Hall, 304 Washington Road, Princeton, NJ 08540, USA}
\email{zhihanw@math.princeton.edu}

\begin{document}

\begin{abstract}
  We show that any minimizing hypercone can be perturbed into one side to a properly embedded smooth minimizing hypersurface in the Euclidean space, and every viscosity mean convex cone admits a properly embedded smooth mean convex self-expander asymptotic to it near infinity.  These two together confirm a conjecture of Lawson \cite[Problem 5.7]{Brothers86_OpenProblems}. 
\end{abstract}
\maketitle

\section{Introduction} \label{Sec_Intro}
  Regularity theory for minimal hypersurfaces has long been studied in the history.  By the landmark work of \cite{DeGiorgi61, FedererFleming60, Simons68, Federer70_DimReduc}, an area-minimizing \textit{hypersurface} in the Euclidean space or in a Riemannian manifold is always smooth away from some closed singular set of Hausdorff codimension $\geq 7$; Near each singularity, minimizing hypersurfaces are modeled on minimizing hypercones, i.e. minimizing hypersurfaces in the Euclidean space invariant under rescalings. 
  This regularity result was generalized to stable minimal hypersurfaces without codimension $1$ edge-type singularities by \cite{SchoenSimon81, Wickramasekera14_RegStableMH}.  Later, \cite{Simon93_SurveyRect, CheegerNaber13_QuantStratif, NaberValtorta20_MS_Rect} carefully studied the structure of singular set, proving that they are in fact codimension $7$-rectifiable.  On the other hand, \cite{Simon2021_Prescip_Sing} constructed examples of stable minimal hypersurfaces in $\RR^{m+7+1}$ ($m\geq 1$) under $C^\infty$ Riemannian metrics arbitrarily close to the Euclidean metric, whose singular set is any prescribed closed subset in $\RR^m\times\{\mathbf{0}\}$.  This suggests that the singular set can be bad in general.  
  
  On the other hand, those singularities of minimizing hypersurfaces seem to be \textit{unstable} under perturbation.  In \cite{HardtSimon85}, for a regular minimizing hypercone $C\subset \RR^{n+1}$ (regular means $C$ has only isolated singularity at the origin), it was shown that there exists a unique (up to rescaling) minimizing hypersurface $S$ lying on one side of $C$ which has \textbf{no singularity}. Moreover, such $S$ is a radial graph over certain domain on $\SSp^n$.  Based on this, \cite{Smale93} shown that for a closed oriented $8$-manifold $M$ with nontrivial $7$-th homology group and a $C^\infty$-generic Riemannian metric $g$, every (homologically) minimizing hypersurface in $(M, g)$ has no singularity.  Later, this was generalized by \cite{ChodoshLioukumovichSpoloar20_GenericReg, LiYyWangZH20_Generic} to generic metric on an arbitrary closed $8$-manifold to construct at least one entirely smooth minimal hypersurface (not necessarily area-minimizing). 
  
  In this paper, we generalize the existence result by Hardt and Simon to arbitrary minimizing hypercones, not necessarily with only isolated singularity.
  \begin{Thm} [c.f. Theorem \ref{Thm_Pf_MinSmooth}] \label{Thm_Intro_MinSmooth}
   Let $E \subset \RR^{n+1}$ be a closed subset such that $C := \partial E$ is a minimizing hypercone in $\RR^{n+1}$. Then there exists $S\subset \Int(E)$ a properly embedded smooth minimizing hypersurface in $\RR^{n+1}$ with distant $1$ to the origin and $C$ to be the tangent cone at infinity. 
   Moreover, $S$ is a radial graph over $\Int(E)\cap \SSp^n$.
  \end{Thm}
  We might expect this theorem to be a first step towards the full generic regularity problem for minimizing hypersurfaces in general dimension. 
  However, we assert that the uniqueness (up to rescaling) of such one-sided smooth perturbation is still open, even in the splitting case, i.e. when $C = C_0 \times \RR^k$ for some regular minimizing hypercone $C_0$.  In \cite{Simon2021_Liouville}, uniqueness in the splitting case is proved under further strict stability assumption.  We also mention that \cite{Lohkamp18_Smoothing} claimed the existence and uniqueness of such one-sided minimizing perturbation, with a seemingly incomplete proof.
  On the other hand, even if the uniqueness (up to rescaling) is established, it's still unclear how to use it to smoothing a general minimizing hypersurface in a Riemannian manifold. 
  
  When $C$ is not minimizing in $E$, by a barrier argument, such minimal hypersurface lying in the interior of $E$ should not exist. \cite{Lin87_Approx} shown that for $n+1\leq 8$, for a regular minimal hypercone $C = \partial E\subset\RR^{n+1}$, there exists a smooth mean convex hypersurface asymptotic to $C$ near infinity.  \cite{Ding20_MinCone_SelfExpander} shown that more generally, if the cone $C$ is $C^{3,\alpha}$-regular and mean convex of arbitrary dimension, but not minimizing in $E$, then there exists a unique smooth strictly mean convex self-expander sitting inside $\Int(E)$ and properly embedded in $\RR^{n+1}$.  Here a self-expander is a hypersurface satisfying the equation, \[
    \vec{H} - \frac{X^\perp}{2} = 0,   \] 
  where $\vec{H}$ is the mean curvature vector and $X^\perp$ is the projection of position vector $X$ onto the normal direction.  Self-expander was introduced in \cite{EckerHuisken89_EntireGraph} as model of long time behavior for mean curvature flow, and also studied by \cite{Ilmanen98_LecturesMCF} as model for mean curvature flow coming out of a singularity. See Section \ref{Subsec_Self-expander} for more on geometry of self-expanders.  
  
  We also generalize Ding's result to arbitrary mean convex hypercone, not necessarily regular.
  \begin{Thm}[c.f. Theorem \ref{Thm_Pf_MeanConv}] \label{Thm_Intro_MeanConvexSmoothing}
   Let $E \subset \RR^{n+1}$ be a closed subset such that $C := \partial E$ is a mean convex hypercone in $\RR^{n+1}$ and not minimizing in $E$. Then there exists $S\subset \Int(E)$ a smooth strictly mean convex self-expander properly embedded in $\RR^{n+1}$ with $C$ to be the tangent cone at infinity.  In particular, $S$ is a radial graph over $\Int(E)\cap \SSp^n$.  
  \end{Thm}
  Here, mean convexity assumption for $E$ is in viscosity sense, see Definition \ref{Def_Viscosity mean convex}. In particular, Theorem \ref{Thm_Intro_MeanConvexSmoothing} applies when $\partial E$ is a stationary minimal hypercone, possibly with singularities, which is not minimizing in $E$.  On the other hand, uniqueness is again unclear in this non-smooth case. \\
  
  Combining Theorem \ref{Thm_Intro_MinSmooth} and \ref{Thm_Intro_MeanConvexSmoothing}, we confirm a conjecture of Lawson \cite[Problem 5.7]{Brothers86_OpenProblems}:
  \begin{Thm} \label{Thm_Intro_LawsonConj}
   Given a stable minimal hypercone $C\subset \RR^{n+1}$ and every $\epsilon>0$, there exists a properly embedded smooth hypersurface with positive mean curvature in $\BB_1^{n+1}(\mathbf{0})$ within Hausdorff distant $\leq\epsilon$ from $C\cap \BB_1^{n+1}(\mathbf{0})$.
  \end{Thm}
  
  A key ingredient of the proof of both Theorem \ref{Thm_Intro_MinSmooth} and \ref{Thm_Intro_MeanConvexSmoothing} is the following divergence of positive super-solution to Jacobi field equation near singularities.
  \begin{Thm} [c.f. Corollary \ref{Cor_pos super Jac field div near sing}] \label{Thm_pos super Jac field div}
   Let $(M, g)$ be a smooth Riemannian manifold (not necessarily complete), $\Lambda>0$.  Let $\Sigma \subset M$ be a two-sided stable minimal hypersurface with singular set of codimension $\geq 7$.  Let $u \in C^2_{loc}(\Sigma)$ be a positive function such that \[
     \Delta_\Sigma u + |A_\Sigma|^2 u - \Lambda u \leq 0,   \]
   on $\Sigma$.  Then for every $x\in \Sing(\Sigma)$, \[
     \liminf_{y\to x} u(y) = +\infty.   \]   
  \end{Thm}
  Such behavior of $u$ has been established by \cite{Simon08} for minimizing hypersurfaces, or more generally, for submanifolds belonging to a \textit{regular multiplicity 1 class}.  A phenomenon of similar spirit was also exploited in \cite{SchoenYau17_PSC} for minimal slicings and by \cite[Lemma 2.14]{WangZH20_deformations} to study stable minimal hypersurfaces lying on one side of a regular cone near infinity.
  In this paper, we prove it for general stable minimal hypersurfaces, following the strategy of \cite{Simon08} but more directly by first proving a Harnack inequality for stable minimal hypersurfaces.
  
\subsection*{Sketch of the Proof of Theorem \ref{Thm_Intro_MinSmooth} and \ref{Thm_Intro_MeanConvexSmoothing}.}  
  Let $C = \partial E\subset \RR^{n+1}$ be a minimizing hypercone.  First one can approximate $E\cap \SSp^n$ from interior by mean convex domains $\cE_j$ with optimal regularity.  Then consider in $\RR^{n+1}$ the Plateau problem of minimizing area among integral currents with boundary $\partial \cE_j$.  Using a similar argument as \cite{HardtSimon85}, one can show that such area-minimizer $S_j$ is a smooth radial graph over $\cE_j$ and (after subtracting the boundary) lies in the interior of the cone over $\cE_j$.  Moreover, when $j\to \infty$, $S_j$ converges to the truncated cone $C\cap \BB_1$.  Hence if one rescale $S_j$ by their distant to the origin, then the rescaled minimizing hypersurfaces will subconverge to some hypersurface $S$ minimizing in $\RR^{n+1}$ and lying in $E$, with distant $1$ to the origin.  To see $S$ is smooth, consider the Jacobi field $\phi = X\cdot\nu_S$ on regular part of $S$ induced by rescaling, here $X$ is the position vector and $\nu_S$ is the unit normal field of $S$ pointing away from the cone.  $S$ being limit of rescaling of radial graphs $S_j$ guarantees that $\phi\geq 0$;  And $\phi$ can't be identically zero since otherwise $S$ must be a cone and can't have distant $1$ to the origin.  Thus by strong maximum principle, $\phi$ is everywhere positive;  Also by definition, $\phi$ is bounded on each ball in $\RR^{n+1}$.  Therefore, Theorem \ref{Thm_pos super Jac field div} implies that the singular set of $S$ is empty.   
   
  When $E$ is viscosity mean convex and not perimeter minimizing in $E$, one seeks to minimizing $\mbfE$-functional introduced by \cite[Lecture 2 C]{Ilmanen98_LecturesMCF} (where it's called $\mbfK$-functional) to find a self-expander $S'$ asymptotic to $\partial E$ near infinity. However, since the rescaling of a self-expander is usually not a self-expander, one can not use maximum principle to conclude that $S'$ is disjoint from its rescalings. Instead, we consider the level set flow $\bigsqcup_{t\geq 0}D_t\times \{t\} \subset \RR^{n+1}\times \RR$ starting from $(\Int(E))^c$. We still first perturb $E\cap \SSp^n$ into a sequence of optimally regular mean convex domain $\cE_j$ in the interior of $E\cap \SSp^n$, and then solve Plateau problem to find self-expanders $S'_j$ lying on one side of the cone $C_j$ over $\partial \cE_j$, and asymptotic to $C_j$ near infinity.  Then using avoidance principle for weak mean curvature flow, we argue that the rescalings $\lambda\cdot S'_j$ are disjoint from $\Int(D_1)$ for every $\lambda\geq 1$. Then by taking $j\to \infty$, we obtain an $\mbfE$-minimizing self-expander $S' \subset E$ asymptotic to $\partial E$ near infinity, and disjoint from $\Int(D_1)$.  Using definition of level set flow, we conclude that $S'$ must coincide with $\partial D_1$, and the rescalings $\lambda\cdot S'$ are still all disjoint from $\Int(D_1)$, $\forall \lambda\geq 1$.  Moreover, $S'\neq \partial E$ since $\partial E$ is not area-minimizing in $E$. These altogether imply that the eigenfunction $\phi:= X\cdot \nu_{S'} = 2H_{S'}$ of Jacobi operator of $\mbfE$-functional on $S'$ induced by rescaling is non-negative and not vanishing identically.  Repeat the process above, we get $\phi>0$ everywhere and $S'$ has no singularity.  
  
\subsection*{Organization of the Paper.} 
  Section \ref{Sec_Pre} contains the basic notations we use in this paper as well as a brief review of geometric measure theory, geometry of self-expanders and weak notions of mean curvature flow.  In Section \ref{Sec_Harnack Ineq}, we prove a multiplicity $1$ result for stable minimal hypersurfaces, which enable us to derive a Neumann-Sobolev inequality and a Harnack inequality for stable minimal hypersurfaces following the same argument as \cite{BombieriGiusti72_Harnack} for minimizing boundaries.  Using this, in Section \ref{Sec_Growth Rate Lower Bd} we prove a more precise asymptotic lower bound for super-solution of Jacobi field equations on a stable minimal hypersurface, and derive Theorem \ref{Thm_pos super Jac field div} as a corollary.  Finally, we state and prove a more concrete version of Theorem \ref{Thm_Intro_MinSmooth}, \ref{Thm_Intro_MeanConvexSmoothing} and finish the proof of Theorem \ref{Thm_Intro_LawsonConj} in Section \ref{Sec_Pf Main Thm}.
  
\subsection*{Acknowledgement}
  I am grateful to my advisor Fernando Cod\'a Marques for his constant support and guidance.

\section{Preliminaries} \label{Sec_Pre}
  Throughout this paper, let $\RR^{n+1}$ be the Euclidean space of dimension $n+1$, $n\geq 1$; Let
  \begin{itemize}
  \item $\BB_r^{n+1}(x)$ be the open ball of radius $r$ in $\RR^{n+1}$ centered at $x$; We may omit the superscript $n+1$ if there's no confusion about dimension; We may write $\BB_r:= \BB_r(\mathbf{0})$ to be the ball centered at the origin $\mathbf{0}$;
  \item $\SSp^n := \partial \BB_1^{n+1}$ be the unit sphere;
  \item $\AAa(x; r,s)$ be the open annuli $\BB_r(x)\setminus \Clos(\BB_s(x))$ centered at $x$;
  \item $C(\cE):= \{r\omega: \omega\in \cE\subset \SSp^n, r\geq 0\}\subset \RR^{n+1}$ be the cone generated by subset $\cE\subset \SSp^n$;
  \item $\eta_{x, r}$ be the map between $\RR^{n+1}$, maps $y$ to $r(y-x)$; We may omit subscript $x$ if $x=\mathbf{0}$;
  \item $\scH^k$ be the $k$-dimensional Hausdorff measure;
  \item $g_{Euc}$ be the Euclidean metric on $\RR^{n+1}$.
  \end{itemize}
  
  For any smooth oriented Riemannian manifold $(M, g)$ of dimension $n+1$, write 
  \begin{itemize}
  \item $\Int(A)\ $ be the interior of a subset $A\subset M$;
  \item $\Clos(A)\ $ be the closure of a subset $A\subset M$;
  \item $\partial A := \Clos(A)\setminus \Int(A)\ $ be the topological boundary of $A$;
  \item $dist_g\ $ be the distant function over $(M, g)$;
  \item $B_r(A):= \{x\in M: dist_g(x, A)<r\}$ be the open $r$-neighborhood of subset $A\subset M$;
  \item $\nabla_g\ $ (or $\nabla^g$) \ be the Levi-Civita connection with respect to $g$;
  \item $\exp^g_x\ $ be the exponential map of $(M, g)$ on the tangent space $T_x M$;
  \item $injrad(x; M, g)\ $ be the injectivity radius of $(M, g)$ at $x$;
  \item $Ric_g\ $ be the Ricci curvature tensor of $(M, g)$;
  \item $\scX_c(U)\ $ be the space of compactly supported smooth vector field on an open subset $U\subset M$;
  \item $e^{tX}\ $ be the one-parameter family of diffeomorphism generated by $X\in \scX_c(U)$.  
  \end{itemize}
  We may omit the super or subscript $g$ if there's no confusion.  Also for two subsets $U, V \subset M$, write $U\subset \subset V$ if $\Clos(U)$ is a compact subset of $V$. 
  
  For a sequence of closed subset $\{E_j\}_{1\leq j\leq \infty}$ of $M$, call $E_j$ converges to $E_\infty$ \textbf{locally in the Hausdorff distant sense}, if for every compact subset $K\subset M$ and every $\epsilon>0$, there exists $j(\epsilon, K)>>1$ such that $\forall j\geq j(\epsilon, K)$, \[
    E_j\cap K \subset B_\epsilon(E_\infty), \ \ \ \ \ E_\infty\cap K \subset B_\epsilon(E_j).   \]
  
  Let $\Sigma \subset M$ be a two-sided hypersurface, two-sided means it admits a global normal field $\nu$.  When $\Sigma$ is portion of the boundary of some specified domain, unless otherwise mentioned, we use the convention that $\nu$ is chosen to be the outward pointed normal field.   In this article, unless otherwise stated, every hypersurface is assumed to be two-sided and \textbf{optimally regular}, i.e. $\scH^{n-2}(\Clos(\Sigma)\setminus \Sigma) = 0$ and $\scH^n\llcorner \Sigma$ is locally finite.  Let
  \begin{itemize}
  \item $\Reg(\Sigma):= \{x\in \Clos(\Sigma): \Clos(\Sigma) \text{ is a }C^{1,1} \text{ embedded hypersurafce near }x\}$ be the regular part of $\Sigma$;
  \item $\Sing(\Sigma):= \Clos(\Sigma) \setminus \Reg(\Sigma)$ be the singular part.
  \end{itemize}
  By adding points to $\Sigma$ if necessary, we may identify $\Sigma = \Reg(\Sigma)$ in this paper.  Call $\Sigma$ \textbf{regular} in an open subset $U$ if $\Sing(\Sigma)\cap U = \emptyset$.  We shall work with the following function spaces on $\Sigma$:
  \begin{itemize}
  \item $L^p(\Sigma)$ measurable functions $f$ with $\int_\Sigma |f|^p < +\infty$;
  \item $W^{1,2}_{loc}(\Sigma)$ measurable functions which restricts to $W^{1,2}$-function on each compact smooth sub-domain in $\Sigma$;
  \item $C^k_{loc}(\Sigma)$ functions on $\Sigma$ which admit up to $k$-th order continuous derivatives, possibly unbounded on $\Sigma$.
  \end{itemize}
  
  Also write 
  \begin{itemize}
  \item $H_\Sigma := -div_\Sigma (\nu)$ be the scalar mean curvature of $\Sigma$, and $\vec{H}_\Sigma:= H_\Sigma\cdot \nu$ be the mean curvature vector;
  \item $A_\Sigma := -\nabla \nu$ be the second fundamental form of $\Sigma$. 
  \end{itemize}
  Note that under this convention, the mean curvature vector does not depend on the choice of normal field, and the scalar mean curvature for unit sphere with respect to the outward pointed normal field is negative.  
  Recall $\Sigma$ is \textbf{minimal} if and only if $H_\Sigma \equiv 0$.
  
  Call hypersurface $\Sigma\subset \RR^{n+1}$ a \textbf{hypercone} if it's invariant under dilation $\eta_\lambda$, $\forall \lambda >0$.

  A minimal hypersurface $\Sigma \subset (M, g)$ is called \textbf{stable} in an open subset $U$, if \[
    \frac{d^2}{ds^2}\Big|_{s=0} \scH^n(e^{sX}(\Sigma))\geq 0, \ \ \ \forall X\in \scX_c(U).   \]
  By \cite{SchoenSimon81}, this is equivalent to that  $\scH^{n-7 + \epsilon}(\Sing(\Sigma)) = \emptyset$, $\forall \epsilon>0$ and that 
  \begin{align}
   Q_\Sigma(\phi, \phi):= \int_\Sigma |\nabla\phi|^2 - (|A_\Sigma|^2 + Ric_g(\nu, \nu))\phi^2 \geq 0, \ \ \ \forall \phi\in C_c^1(\Sigma\cap U).  \label{Equ_2nd Variation for area in general}
  \end{align}
  Let 
  \begin{align}
   L_\Sigma := \Delta_\Sigma + |A_\Sigma|^2 + Ric_g(\nu, \nu),  \label{Equ_Jac oper in general}
  \end{align}
  be the Euler-Lagrangian operator associated to $Q_\Sigma$, known as the \textbf{Jacobi operator}.  Every $u\in C^2_{loc}(\Sigma)$ solving $L_\Sigma u = 0$ on $\Sigma$ is called a \textbf{Jacobi field}.

  \subsection{Basics in geometric measure theory} \label{Subsec_Basic GMT}
   We recall some basic notions from geometric measure theory and refer the readers to \cite{Simon83_GMT, Federer69} for details.  For $1\leq k\leq n+1$, in an $n+1$ dimensional manifold $(M, g)$ (not necessarily complete, isometrically embedded in some $\RR^L$ if necessary), write
   \begin{itemize}
   \item $\mbfI_k(M)$ be the space of integral $k$ currents on $M$;
   \item $\cZ_k(M):= \{T\in \mbfI_k(M): \partial T = 0\}$ be the space of integral $k$-cycles, where $\partial: \mbfI_k \to \mbfI_{k-1}$ be the boundary operator;
   \item $\cI\cV_k(M)$ be the space of integral $k$ varifolds on $(M, g)$;
   \item $\spt(V)$ be the support of a varifold $V\in \cI\cV_k(M)$;
   \item $f_\sharp$ be the push forward of currents or varifolds associated to a proper Lipschitz map $f$. 
   \end{itemize}
   For a hypersurface $\Sigma\subset (M, g)$, let $|\Sigma|_g\in \cI\cV(M)$ be the integral varifold associated to $\Sigma$; If further $\Sigma$ is oriented, denote $[\Sigma]\in \mbfI_n(M)$ to be the associated integral current.  More generally, for $T\in \mbfI_n(M)$, denote $|T|_g$, $\|T\|_g$ to be the associated integral $n$-varifold and Radon measure on $M$ correspondingly. The metric subscript $g$ above will be omit if there's no confusion.
   
   For an open subset $U\subset M$, let $\cF_U$ and $\mbfF_U$ be the flat metric on $\mbfI_n(M)$ and varifold metric on $\cI\cV_n(M)$ in $U$, and let $\mbfM_U$ be the mass for an integral current;  Omit $U$ if $U = M$. 
   
   For a smooth vector field $X\in \scX_c(U)$ and $V\in \cI\cV_n(U)$, the first variation of the area of $V$ with respect to $X$ is \[
     \delta V(X) = \frac{d}{dt}\Big|_{t=0}\|V\|(U) = \int div^\pi X\ dV(x, \pi),   \]
   Recall $V$ is called \textbf{stationary} in $U$ if and only if $\delta V(X) =0$, $\forall X\in \scX_c(U)$.
   By constructing appropriate test vector fields, whenever $V$ is stationary in $M$, for every $x\in M$, the following \[
     r\mapsto \theta_g(x, r; \|V\|):= \frac{e^{n\lambda r}}{\omega_n r^n}\|V\|(B_r(x))   \]
   is monotone non-decreasing in $r\in (0, injrad(x; M, g))$, where \[
     \lambda = \lambda(x, r):= \inf \{C\geq 0: \nabla_g^2(\rho_x^2) \geq 2(1-C\rho_x^2)\text{ as a quadratic form, on }B_r(x)\};  \]
   where $\rho_x:= dist_g(\cdot, x)$ is the distant function to $x$.   
   And when $g = g_{Euc}$, $\theta_g(x, r; \|V\|)$ is constant in $r$ if and only if $V$ is a cone.  
   Also, denote $\theta(x; \|V\|):= \lim_{r\searrow 0} \theta_g(x, r; \|V\|)$ to be the density of $V$ at $x$.\\

   Following \cite{DeGiorgi61, Simon83_GMT}, call a Lebesgue measurable subset $P\subset M$ a \textbf{Caccioppoli set}, if for every open subset $W\subset \subset M$ and every smooth vector field $X\in \scX_c(W)$, there exists some constant $C(P, W)$ such that \[
     |\int_P div (X)\ d\scH^{n+1}| \leq C(P, W)\|X\|_{C^0}.   \]
   By \cite{DeGiorgi61, Simon83_GMT}, a Caccioppoli set $P$ is naturally an integral $n+1$-current $[P] \in \mbfI_{n+1}(M)$. $\|\partial [P]\|(W)$ is called the \textbf{perimeter} of $P$ in $W$, and is usually denoted as $\cP(P; W)$.  Also define the $\mbfF$-metric among Caccioppoli sets by \[
     \mbfF(P_1, P_2):= \cF(\partial [P_1], \partial [P_2]) + \mbfF(\|\partial [P_1]\|, \|\partial [P_2]\|).   \]
   
   When $E\subset M$ is a measurable subset, call a Caccioppoli set $P\subset M$ (homologically) perimeter minimizing in $E$ if $P \subset E$ and for every open subset $W'\subset \subset W\subset M$ and every Caccioppoli set $Q\subset M$ with $P\Delta Q\subset E\cap W'$, we have \[
     \cP(P; W) \leq \cP(Q; W).   \]
   If $E\subset M$ is a closed subset, $P \subset E$ is minimizing in $E$ and $|\partial [P]|$ is stationary in $M$, then the regular part of $\partial [P]$ is a multiplicity one and stable; Being minimizing in $E$ guarantees that there's no singularity of $\partial [P]$ near which $\spt(\partial [P])$ is a union of embedded $C^{1,\alpha}$ hypersurface-with-boundary meeting along their common boundary.  Hence, by \cite[Theorem 18.1]{Wickramasekera14_RegStableMH}, $\Reg(\partial [P])$ is a stable minimal hypersurface in $M$ with optimal regularity. 
   
   The compactness of perimeter minimizing Caccioppoli sets in closed subsets is a little bit subtle, due to the lack of direct cut-pasting argument as in \cite[Theorem 34.5]{Simon83_GMT}. 
   \begin{Lem}[Compactness] \label{Lem_Cptness of perimeter minzer in E_j}
    Let $\{E_j\}_{1\leq j\leq \infty}$ be a sequence of closed subset in $M$ such that $E_j \to E_\infty$ and $\partial E_j \to \partial E_\infty$ both in locally Hausdorff distant sense in $M$. Let $\{g_j\}_{1\leq j\leq \infty}$ be a sequence of smooth metric on $M$ such that $g_j$ smoothly converges to $g_\infty$.
    For $1\leq j<\infty$, let $P_j\subset E_j$ be a Caccioppoli set in $M$ minimizing perimeter in $E_j$, with $|\partial [P_j]|$ stationary in $M$.  

    Then there exists some Caccioppoli set $P_\infty \subset E_\infty$ with $\Reg(\partial [P_\infty])$ an optimally regular stable minimal hypersurface in $M$, such that after passing to a subsequence of $j\to \infty$, $\mbfF(P_j, P_\infty) \to 0$.
    
    Moreover, if $E_j \equiv E_\infty$, then $P_\infty$ is also perimeter minimizing in $E_\infty$.
   \end{Lem}
   \begin{proof}
    For each smooth domain $W\subset\subset M$, by comparing the perimeter of $P_j$ with $P_j\setminus W$, we know that the $\limsup \|\partial [P_j]\|(W) <+\infty$.  Hence by Fleming-Federer compactness Theorem, $\partial [P_j]$ subconverges in flat topology to some boundary of Caccioppoli set $\partial [P_\infty]$; And by \cite{SchoenSimon81}, the stable minimal hypersurfaces $|\partial [P_j]|$ subconverges to some stationary integral varifold $V_\infty$ supported on a stable minimal hypersurace in $E_\infty$.  
    
    Moreover, $V_\infty$ must be of multiplicity $1$, since otherwise, by \cite[Theorem 1]{SchoenSimon81}, near some regular point $x$ of higher multiplicity, $\|\partial [P_j]\|$ are multiple graphs over some hyperplane through $x\in \spt(V_\infty) \subset E_\infty$. When $x\in \Int(E_\infty)$, since $\partial E_j$ converges to $\partial E_\infty$ locally in the Hausdorff distant sense, this is impossible by the local perimeter-minimizing property of $P_j$ near $x$; When $x\in \partial E_\infty$, we know that for $j>>1$, there must be two adjacent graphs bounding a slab in $P_j$ near $x$.  This violates that $P_j$ is perimeter-minimizing in $E_j$ by subtracting a small ball from this slab.
    
    Using Schoen-Simon's epsilon regularity theorem \cite[Theorem 1]{SchoenSimon81}, $V_\infty$ being of multiplicity $1$ implies that $V_\infty = |\partial [P_\infty]|$, hence $P_j$ subconverges to $P_\infty$ in $\mbfF$-metric.
    
    When $E_j \equiv E_\infty$, the cut-pasting argument in \cite[Theorem 34.5]{Simon83_GMT} works here to show that $P_\infty$ is perimeter minimizing in $E_\infty$.
   \end{proof}

   \begin{Def} \label{Def_Viscosity mean convex}
    Call a closed subset $E\subset M$ \textbf{viscosity mean convex}, if $E = \Clos(\Int(E))$ and for every smooth open domain $\Omega \subset E$ with $x\in \partial \Omega\cap \partial E$, the mean curvature with respect to the outward pointed normal field of $\partial \Omega$ at $x$ is $H_{\partial \Omega}(x)\leq 0$.
    
    Call a closed viscosity mean convex subset $E\subset M^{n+1}$ \textbf{$C^{1,1}$-optimally regular}, if there exists a closed subset $\cS \subset \partial E$ with Hausdorff dimension $\leq n-7$, such that $\partial E \setminus S$ is a $C^{1,1}$ embedded hypersurface and $\forall x\in \cS$, $\partial E$ is a stable minimal hypersurface with optimal regularity near $x$.  Call the smallest such $\cS \subset \partial E$ the \textbf{singular set} of $\partial E$.
   \end{Def}
   Without assuming $E = \Clos(\Int(E))$, the notion of viscosity mean convex was introduced and studied in \cite[Definition 3.1]{IlmanenSternbergZiemer98}, where it was called \textit{barrier for minimal surface equation}. 
   \begin{Rem} \label{Rem_Viscosity mean convex}
   \begin{enumerate} [(1)]
   \item  By \cite{SolomonWhite89_Maxim}, the support of every Caccioppoli set in $M$ with stationary boundary is viscosity mean convex.  In particular, every stable minimal hypercone in $\RR^{n+1}$ bounds a viscosity mean convex subset.
   \item  It follows directly from the definition that if $E\subset M$ is a viscosity mean convex closed subset, then for every connected component $U$ of $\Int(E)$, $\Clos(U)$ is also viscosity mean convex.  Hence, in the following discussion, we may additionally assume that a viscosity mean convex subset has connected interior.
   \item  When a closed viscosity mean convex set $E\subset M$ is $C^{1,1}$-optimally regular, on its regular part, we can define the mean curvature vector $\vec{H}_{\partial E} \in L^\infty(\partial E)$ almost everywhere. And $E$ being viscosity mean convex guarantees that $\langle\vec{H}_{\partial E}, \nu \rangle \leq 0$, where $\nu$ is the outward pointed normal field.   
   In Appendix \ref{Sec_Minz w obstacle}, it is shown that any compact viscosity mean convex subset can be approximated by $C^{1,1}$-optimally regular mean convex subset.
   \end{enumerate}
   \end{Rem}
   
  \subsection{Self-expanders} \label{Subsec_Self-expander}
  For a hypersurface $\Sigma\subset \RR^{n+1}$ and a bounded open subset $W\subset \RR^{n+1}$, define the $\mbfE$-functional of $\Sigma$ in $W$ to be, \[
    \mbfE[\Sigma; W]:= \int_W e^{|x|^2/4}\ d\|\Sigma\|(x).   \]
   
  Such $\mbfE$-functional is introduced in \cite[Lecture 2 C]{Ilmanen98_LecturesMCF} in which it's named $\mbfK$-functional, and frequently studied by \cite{BernsteinWang18_Degree_SelfExpander, BernsteinWang19_RelEntropy_SelfExpander, BernsteinWang19_TopUniq_SelfExpander, BernsteinWang2020_Minmax_SelfExpander, BernsteinWang21_Cptness_SelfExpander, BernsteinWang21_Space_SelfExpander} and by \cite{Ding20_MinCone_SelfExpander}. Here we keep the notations with Bernstein-Wang.  Critical points $\Sigma$ of $\mbfE$ are known to be \textbf{self-expanders}, i.e. it satisfies the following equation \[
    \vec{H}_\Sigma - \frac{X^\perp}{2} = 0,   \]
  where $X^\perp$ denotes the projection of position vector onto normal direction of $\Sigma$. Equivalently, by \cite{EckerHuisken89_EntireGraph}, $\{\eta_{\sqrt{t}}(\Sigma)\}_{t>0}$ is a family of hypersurfaces flowing by mean curvature.

  Clearly, the $\mbfE$-functional can be extended to any Radon measure $\mu$ on $\RR^{n+1}$, \[
    \mbfE[\mu; W]:= \int_W e^{|x|^2/4}\ d\mu(x).   \]
  We also denote for simplicity that for a Caccioppoli sets $P\subset \RR^{n+1}$, $\mbfE[P; W]:= \mbfE[\|\partial [P]\|; W]$ and for an $n$-varifold $V$, $\mbfE[V; W]:= \mbfE[\|V\|; W]$.
  
  For a bounded open subset $W\subset \RR^{n+1}$, call a Caccioppoli set $P \subset \RR^{n+1}$ \textbf{$\mbfE$-minimizing} in $W$, if for every Caccioppoli set $Q\Delta P \subset \subset W'\subset \subset W$, we have \[
    \mbfE[Q; W'] \geq \mbfE[P; W'].   \]
  And call an integral $n$-varifold $V\in \cI\cV_n(W)$ $\mbfE$-\textbf{stationary} in $W$, if for every $X\in \scX_c(W)$,  \[
    \delta\mbfE[V; W](X):= \frac{d}{dt}\Big|_{t=0} \mbfE[e^{tX}{ }_\sharp V; W] = 0;   \] 
  Call $V$ $\mbfE$-\textbf{stable} if it's $\mbfE$-stationary and for every $X\in \scX_c(W)$, \[
    \delta^2\mbfE[V; W](X, X):= \frac{d^2}{dt^2}\Big|_{t=0} \mbfE[e^{tX}{ }_\sharp V; W] \geq 0.   \]
  Note that for an $n$-varifold, $\mbfE$-functional is the $n$-th area-functional under the metric $\hat{g}:= e^{|x^2|/2n}\cdot g_{Euc}$ on $\RR^{n+1}$. Hence, being $\mbfE$-stationary (resp. $\mbfE$-stable, $\mbfE$-minimizing) is equivalent to being stationary (resp. stable, minimizing) with respect to the area functional under $\hat{g}$.
  
  For a 2-sided hypersurface $\Sigma\subset W$, by \cite[(3.16)]{Ding20_MinCone_SelfExpander}, being $\mbfE$-stable is equivalent to that for every $\phi\in C_c^1(\Sigma\cap W)$, 
  \begin{align}
   \int_W [|\nabla \phi|^2 - (|A_\Sigma|^2-\frac{1}{2})\phi^2]\cdot e^{|x|^2/4}d\|\Sigma\|(x) \geq 0.   \label{Equ_Stability Ineq for mbfE}
  \end{align}
  The Euler-Lagrangian operator of the quadratic form on the left hand side of (\ref{Equ_Stability Ineq for mbfE}) is,
  \begin{align}
   \cL_\Sigma := \Delta_\Sigma + \frac{X}{2}\cdot \nabla_\Sigma + |A_\Sigma|^2 -\frac{1}{2}.  \label{Equ_Jac oper wrt mbfE-functional}
  \end{align}
  Moreover,  if denote $\hat{\Delta}_\Sigma, \hat{A}_\Sigma$ to be the laplacian and second fundamental form of $\Sigma$ under metric $\hat{g}$, then a simple computation shows that for every function $u$ on $\Sigma$,
  \begin{align}
   \cL_\Sigma u = e^{\frac{|x|^2}{4n}}\cdot\big(\hat{\Delta}_\Sigma + |\hat{A}_\Sigma|^2 - e^{-\frac{|x|^2}{2n}}(1+\frac{n-1}{4n^2}|X^\Sigma|^2)\big)(e^{\frac{|x|^2}{4n}}u),  \label{Equ_Jac oper of mbfE vs Jac oper under hat(g)}
  \end{align}
  where $X^\Sigma$ is the projection of the position vector $X$ onto $T\Sigma$. \\

  Let $\cE\subset \SSp^n$ be a closed subset; Let $E = C(\cE)$ be the closed cone over $\cE$ in $\RR^{n+1}$.  Call a closed set $P\subset \RR^{n+1}$ \textbf{distant asymptotic to $E$} near infinity if locally in the Hausdorff distant sense, when $R\to +\infty$, $\eta_{1/R}(P) \to E$ and $\eta_{1/R}(\partial P) \to \partial E$.
  
  For a unit vector $v\in \SSp^n$ and $\alpha\in (0, \pi)$, define \[
   \cC(v; \alpha):= \{x\in \RR^{n+1}\setminus \{\mathbf{0}\}: \langle x/|x|, v \rangle > \cos \alpha \},    \]
  be the open cone in $v$ direction with open angle $\alpha$. $O(n)$-invariant self-expanders distant asymptotic to $\Clos(\cC(v; \alpha))$ are constructed in \cite[Lemma 2]{AngenentIlmanenChopp95_Eg}, and also studied using variational approach in \cite[Section 4]{Ding20_MinCone_SelfExpander}.  We summarize the results here.  By a rotation, WLOG $v= e_{n+1}$.
  \begin{Thm} \cite{AngenentIlmanenChopp95_Eg, Ding20_MinCone_SelfExpander} \label{Thm_Rot sym self-expander}
   For each $\alpha \in (0, \pi/2)$, there exists a unique $O(n)$-symmetric convex closed smooth domain $F(e_{n+1}; \alpha) \subset \cC(e_{n+1}, \alpha)$ minimizing $\mbfE$-functional in $\RR^{n+1}$ and distant asymptotic to $\Clos(\cC(e_{n+1}; \alpha))$ near infinity.  
   
   Moreover, $F(e_{n+1}, \alpha)$ is an $O(n)$-symmetric super graph, i.e. there exists $u_\alpha \in C^\infty([0, +\infty))$ such that \[
     F(e_{n+1}, \alpha) = \{(x, z): x\in \RR^n, z\geq u_\alpha(|x|)\},   \]
   and $u_\alpha$ satisfies 
   \begin{enumerate} [(i)]
   \item $u_\alpha(0)>0$, $u_\alpha'(0) = 0$, $u_\alpha'(r), u_\alpha''(r)>0$ for $r>0$ and $u(r)/r \to \cot \alpha$ when $r\to +\infty$;
   \item $u_\alpha$ is continuous and strictly monotone in $\alpha$ and for fixed $r\geq 0$, \[
    u_\alpha(r)\to +\infty \ \text{ when }\alpha \to 0, \ \ \ \ \  u_\alpha(r)\to 0 \ \text{ when }\alpha \to \pi/2.   \]
    In particular, $\{\partial F(e_{n+1}, \alpha)\}_{\alpha \in (0, \pi/2)}$ foliates $\RR^n\times (0, +\infty)$.
   \end{enumerate}
  \end{Thm}
  \begin{proof}
   For each $\alpha\in (0, \pi/2)$, in \cite{AngenentIlmanenChopp95_Eg} and \cite{Ding20_MinCone_SelfExpander}, a unique $O(n)$-invariant entire convex graph $S_\alpha\subset \cC(e_{n+1}, \alpha)$ solving self-expanding equation and asymptotic to $\partial \cC(e_{n+1}; \alpha)$ is constructed, where the graphical function $u_\alpha$ satisfies (i). 
   Moreover, the uniqueness of $S_\alpha$ guarantees they varies continuously in $\alpha$.  Let $F(e_{n+1}, \alpha)$ be the closed domain above $S_\alpha$.
   
   As an entire graph, by \cite[Lemma 3.1]{Ding20_MinCone_SelfExpander}, $S_\alpha$ are $\mbfE$-minimizing in $\RR^{n+1}$. Since when $\alpha\searrow 0$, $\Clos(\cC(e_{n+1}, \alpha))\to \{r\cdot e_{n+1}: r\geq 0\}$ locally in the Hausdorff distant sense, by compactness of minimizing boundary we must have, $dist(\mathbf{0}, S_\alpha)\to +\infty$ when $\alpha \searrow 0$, hence for each $r\geq 0$, $u_\alpha (r)\to +\infty$ when $\alpha \searrow 0$.  This also implies that for each $\alpha\in (0, \pi/2)$, when $\alpha'$ sufficiently close to $0$, $S_{\alpha'}\subset  F(e_{n+1}, \alpha)$.  Let $\bar{\alpha}:= \sup\{\beta\in (0, \alpha]: S_t\subset F(e_{n+1}, \alpha) \text{ for every }t\in (0, \beta)\}$. By strong maximum principle for $\mbfE$-functional, $\bar{\alpha} = \alpha$ and $S_t\subset \Int F(e_{n+1},\alpha)$ whenever $t\in (0, \alpha)$.  In other words, $u_\alpha$ is strictly monotone in $\alpha$.
   
   Note that (i) in the Theorem also implies that $0\leq u_\alpha ' < \cot\alpha$ on $[0, +\infty)$ and is monotone increasing.  Hence when $\alpha\nearrow \pi/2$, $u_\alpha$ converges to a constant function whose graph is also a self-expander, hence must be $0$.  This finishes the proof of (ii).
  \end{proof}
  
  For every unit vector $v\in \SSp^n$ and $\alpha\in (0, \pi/2)$, let $T\in O(n+1)$ be an orthogonal transform mapping $e_{n+1}$ to $v$, and denote $F(v; \alpha):= T(F(e_{n+1}, \alpha)) \subset \cC(v; \alpha)$ to be the domain with $\mbfE$-minimizing boundary distant asymptotic to $\cC(v; \alpha)$ near infinity, where $F(e_{n+1}, \alpha)$ is given by Theorem \ref{Thm_Rot sym self-expander}.
  This family of self-expanders are good barriers for general asymptotic conic self-expanders.
  
  \begin{Cor} \label{Cor_Two-side barrier for self-expander}
   Let $\cE\subset \SSp^n$ be a closed subset such that $\Clos(\Int(\cE)) = \cE$, let $0<R\leq +\infty$.  Suppose $F = \spt[F]\subset \RR^{n+1}$ is a closed subset distant asymptotic to $C(\cE)$ near infinity with $|\partial [F]|$ being $\mbfE$-stationary in $(\RR^{n+1}\setminus \partial C(\cE))\cup \BB_R$ and coincide with $C(\cE)$ outside $\BB_R$, i.e. $F\setminus \BB_R = C(\cE)\setminus \BB_R$. (We use the convention that $\BB_R = \RR^{n+1}$ if $R=+\infty$.)
   Then, 
   \begin{align*}
   \begin{cases}
    F(v, \alpha) \subset F, &\ \text{ if }\ \cC(v; \alpha)\cap \SSp^n \subset \cE;\\
    F(v, \alpha) \cap F = \emptyset, &\ \text{ if }\ \Clos(\cC(v; \alpha)\cap \SSp^n) \cap \cE = \emptyset.
   \end{cases}
   \end{align*}
  \end{Cor}
  \begin{proof}
   We prove for $R<+\infty$; The case when $R = +\infty$ will be similar.
   Since $F\setminus \BB_R = C(\cE)\setminus \BB_R$, by Theorem \ref{Thm_Rot sym self-expander} (ii), for sufficiently small $\beta\in (0, \alpha)$, $F(v, \beta)\subset F$.  Consider \[
     \bar{\alpha}:= \sup\{\beta\in (0, \alpha): F(v, \beta)\subset F\},   \]
   If $\bar{\alpha}<\alpha$, then by continuity of $F(v,\beta)$ in $\beta$ from Theorem \ref{Thm_Rot sym self-expander},  $F(v, \bar{\alpha}) \subset F$ and $\partial F(v, \bar{\alpha})\cap \spt(\|\partial [F]\|) \neq \emptyset$.  Since $|\partial [F]|$ and $|\partial F(v, \bar{\alpha})|$ are both $\mbfE$-stationary in $(\RR^{n+1}\setminus \partial C(\cE))\cup \BB_R$, by strong maximum principle \cite{SolomonWhite89_Maxim}, $\partial F(v, \bar{\alpha})\subset\spt(\partial [F])$, contradicts to that $\Clos(\cC(v; \bar{\alpha}))\setminus \BB_R \subset \Int(F)\setminus \BB_R$.   Hence $\bar{\alpha} = \alpha$ and $F(v, \alpha)\subset F$. 
   
   Similarly, when $\Clos(\cC(v; \alpha)\cap \SSp^n) \cap \cE = \emptyset$, the argument above provides,  \[
     \sup\{\beta\in (0, \pi/2): F(v, \beta)\cap F = \emptyset\} >\alpha.   \]
   Hence by Theorem \ref{Thm_Rot sym self-expander}, $F(v, \alpha)\cap F =\emptyset$.
  \end{proof}
  
  Note that using Corollary \ref{Cor_Two-side barrier for self-expander}, one can also prove the existence of $\mbfE$-minimizing closed subset distant asymptotic to $C(\cE)$ near infinity, as is sketched in \cite[Lecture 2 F]{Ilmanen98_LecturesMCF} and proved for $C^2$-domains $\cE$ in \cite[Theorem 6.3]{Ding20_MinCone_SelfExpander}.  
  
  In Section \ref{Sec_Pf Main Thm}, we need the following existence result with constraint.
  \begin{Lem} \label{Lem_Exist Self-Expander w Obstacle}
   Let $E\subset \RR^{n+1}$ be a closed $C^{1,1}$-optimally regular mean convex cone, which is not perimeter minimizing itself in $E$.  
   Then there exists a closed Caccioppoli set $F\subset \Int(E)$ such that $F = \spt[F]$ is distant asymptotic to $E$ near $\infty$, and that $\partial [F]$ minimizes $\mbfE$-functional among \[
     \{\partial [Q]: Q\Delta F \subset E \}.   \]
   In particular, $\partial [F]$ is a self-expanding integral cycle with optimal regularity.
  \end{Lem}
  \begin{proof}
   For $R\geq 1$, minimize $\mbfE$-functional in $\BB_{2R}$ among Caccioppoli sets  \[
     \{Q\subset \RR^{n+1}: E\setminus \BB_R \subset Q\subset E\}   \]
   to find a Caccioppoli set $F_R\subset E$.  Since $\partial \BB_R$ and $\Reg(\partial E)$ are both expander mean convex (i.e. $\vec{H} - X^\perp/2$ either vanishes or points inward), by \cite[Chapter 2]{Lin85}, $|\partial [F_R]|$ is $\mbfE$-stationary in $(\RR^{n+1}\setminus \partial E)\cup \BB_R$. WLOG $F_R = \spt(F_R)$.  
   When $R\to +\infty$, by Lemma \ref{Lem_Cptness of perimeter minzer in E_j} and Corollary \ref{Cor_Two-side barrier for self-expander}, $F_R$ subconverges in $\mbfF$-metric to some closed Caccioppoli set $F\subset E$ which is distant asymptotic to $E$ near infinity, $\mbfE$-minimizing in $E$ and $\mbfE$-stationary in $\RR^{n+1}$.  In particular, $\partial F$ is optimally regular.
   
   By strong maximum principle \cite{SolomonWhite89_Maxim, Ilmanen96}, since $F \neq E$ (Otherwise $E$ is $\mbfE$-minimizing in $E$, and by blowing up $E$ at $\mathbf{0}$, $E$ is perimeter minimizing in $E$, which is a contradiction), we must have $F\subset \Int(E)$. 
  \end{proof}
  
  \subsection{Weak set flow and level set flow} \label{Subsec_Weak flow}
   Level set flow in the Euclidean space is introduced by \cite{EvansSpruck91} and independently by \cite{ChenGigaGoto91_LSF}, and is later generalized to defined on manifolds by \cite{Ilmanen92_LSF, Ilmanen93_IntroLSF}.  Here, we shall use the notion of weak (set) flow and level set flow possibly with boundary, developed by \cite{White95_WSF_Top}.
   \begin{Def} \cite{White95_WSF_Top} 
    Let $M$ be a compact $n$-manifold, possibly with boundary, $f: M\times [a, b]\to \RR^{n+1}$ be a continuous map. Call $\cM:= \{(f(t, x), t): t\in [a, b], x\in M\}\subset \RR^{n+1}\times \RR$ a \textbf{classical flow} if, 
    \begin{itemize}
    \item $f$ is smooth on $\Int(M)\times (a, b]$;
    \item $f(\cdot, t)$ is one-to-one on $M$ for each $t\in [a, b]$ and is a smooth embedding on $\Int(M)$ for each $t\in (a, b]$;
    \item For each $(x, t)\in \Int(M)\times (a, b]$, we have \[
      (\frac{\partial }{\partial t}f(x, t))^\perp = \vec{H}(x, t),  \]
      where $^\perp$ means projection onto normal direction of $f(M, t)$, and $\vec{H}(x,t)$ denotes the mean curvature vector of $f(M, t)$ at $f(x, t)$.
    \end{itemize}
    Call \[
      \partial_h \cM:= \{(f(x, t), t): \text{either }t=a \text{ and }x\in M, \text{ or }t\in [a, b] \text{ and }x\in \partial M\},   \]
    the \textbf{heat boundary} of $\cM$.
    
    Given a closed subset $\Gamma\subset \RR^{n+1}\times \RR_+$, call a closed subset $\cM\subset \RR^{n+1}\times \RR_+$ a \textbf{weak (set) flow} generated by $\Gamma$ at $t=0$, if 
    \begin{itemize}
    \item $\cM(0) = \Gamma(0)$;
    \item For every classical flow $\cM'$ with $\partial_h\cM' \cap \cM = \emptyset$ and $\cM' \cap \Gamma = \emptyset$, we always have $\cM' \cap \cM$.
    \end{itemize}
   \end{Def}
   In this paper, we need the following
   \begin{Thm} \cite[Theorem 4.1]{White95_WSF_Top}  \label{Thm_White_Avoid}
    Let $i=1,2$, $\cM_i$ be weak flow generated by $\Gamma_i$ at $t=0$.  Suppose for each $\tau\geq 0$, $\{(t, x)\in \cM_a: t\leq \tau\}$ is compact.
    Suppose also that at some $T>0$, \[
      dist_T(\cM_1, \cM_2) < min_{i\neq j}dist_T(\cM_i, \Gamma_j).   \]
    Then $dist_t(\cM_1, \cM_2) \geq dist_T(\cM_1, \cM_2)$ for $t$ on some interval $[T, T+\epsilon]$.
   \end{Thm}
   Here $dist_t(\cM_1, \cM_2):= dist_{\RR^{n+1}}(\cM_1(t), \cM_2(t))$, where $\cM_i(t):= \{x\in \RR^{n+1}: (x, t)\in \cM_i\}$.
   In particular, this Theorem implies that whenever $\cM_1\cap \Gamma_2 = \cM_2\cap \Gamma_1 = \emptyset$ in Theorem \ref{Thm_White_Avoid}, we always have $\cM_1\cap \cM_2 = \emptyset$.
   
   \cite{White95_WSF_Top} also provide the construction of level set flow under this definition.   
   \begin{Prop} \cite[Proposition 5.1]{White95_WSF_Top} \label{Prop_White_Exist LSF}
    Given a closed subset $\Gamma\subset \RR^{n+1}\times \RR_+$, there's a unique weak flow $\cM$ generated by $\Gamma$ at $t=0$ (known as the \textbf{level set flow}) such that it contains every weak flow generated by $\Gamma$ at $t=0$.
   \end{Prop}
  
   We end up this section with an example.  
   \begin{Lem} \label{Lem_Trancated self-expander is a weak flow}
    Let $W\subset \RR^{n+1}$ be an open subset; $\Sigma\subset W$ be an $\mbfE$-stable hypersurface with optimal regularity. Then for every $W'\subset \subset W$ and every $\lambda\geq 0$, \[
      \cM:= \bigsqcup_{t\geq \lambda} \eta_{\sqrt{t}}(W'\cap \Clos(\Sigma) )\times \{t-\lambda\},   \]
    is a weak (set) flow generated at $t=0$ by \[
      \Gamma= \eta_{\sqrt{\lambda}}(W'\cap \Clos(\Sigma) )\times \{0\} \cup \bigsqcup_{t\geq \lambda} \eta_{\sqrt{t}}(\partial W'\cap \Clos(\Sigma) )\times \{t-\lambda\}.   \]
   \end{Lem}
   \begin{proof}
    Since the restriction of weak flow onto a sub-interval is also a weak flow, we can assume WLOG that $\lambda = 0$.
    Suppose for contradiction that there exists a classical flow $\cM'$ over $[a, b]\subset (0, +\infty)$ such that $\cM'\cap \Gamma = \partial_h \cM' \cap \cM = \emptyset$ but $\cM' \cap \cM \neq \emptyset$.  By definition, \[
      \mathbf{d}(t):= dist_{\RR^{n+1}}(\cM(t), \cM'(t)),   \]
    is a continuous function in $t\in [a, b]$ and vanishes at some point in $(a, b]$; And \[
      \mathbf{d}_\partial (t):= min\{dist_{\RR^{n+1}}(\cM(t), \partial_h\cM'\ (t)), dist_{\RR^{n+1}}(\Gamma(t), \cM'(t))\},   \]
    has infimum $>0$, and $\mathbf{d}(a) = \mathbf{d}_\partial (a)$. 
    Let \[
      T:= \inf\{s\in [a, b]: \mathbf{d}(s)=0\} \in (a, b],   \]
    Then since $\mathbf{d}_\partial (T)>0$ and $\Int(\cM'(t))$ is smooth, we know that $\cR:= \cM(T)\cap \cM'(T)\neq \emptyset$ is a closed subset contained in $\Reg(\cM(T))\setminus \Gamma(T) = \eta_{\sqrt{T}}(\Int(W')\cap \Sigma)$.  Take a smooth closed neighborhood $\cN\supset \eta_{1/\sqrt{T}}(\cR)$ in $\Sigma\cap \Int(W')$, we know that $\sqcup_{t\in [a, T]}\eta_{\sqrt{t}}(\cN)\times \{t\}$ and restriction of $\cM'$ on $[a, T]$ are two classical flows over $[a, T]$ which intersects along $t=T$ but not intersects along the heat boundary. This violates the maximum principle for classical flows \cite[Lemma 3.1]{White95_WSF_Top}.
   \end{proof}
   Note that a more general result can be proved for integral Brakke motions with codimension $1$ restricted to a spacetime closed subset, using the same argument as \cite[10.5]{Ilmanen94_EllipReg}.  But we don't need it in this paper.

\section{Harnack Inequality for Minimal Hypersurfaces} \label{Sec_Harnack Ineq}
  By \cite{SchoenSimon81}, the convergence of stable minimal hypersurfaces might result in multiplicity in the limit.  The goal of this section is to show that if a stable minimal hypersurface is away from higher multiplicity at scale $1$, then so is it in every smaller scales. $n\geq 3$. 
  
  Denote $\scM^{\geq 2}$ to be the space of hypersurfaces with at least one piece of higher multiplicity , more precisely,
  \begin{align*}
    \scM^{\geq 2} := \Big\{V\in \cI\cV_n(\BB_4^{n+1}): &\ V = \sum_{i\geq 1} m_i |\Sigma_i| \text{ for some disjoint minimal hypersurfaces }\\ 
    &\ \Sigma_i\subset (\BB_4, g) \text{ and } m_i\in \NN \text{ such that } \|g - g_{Euc}\|_{C^{3, 1/2}}\leq 1/10, \\
    &\ m_1\geq 2 \text{ and }\Sigma_1\cap \Clos(\BB_2) \neq \emptyset\ \Big\}.    
  \end{align*}
  \begin{Prop} \label{Prop_Multi 1 in smaller scale}
   Let $\Lambda>0$, $\epsilon>0$. There exists $\delta_1(n, \Lambda)>0$ and $\Psi(\epsilon| n, \Lambda)\in (0, 1)$ with the following property. 
   If $\Sigma \subset (\BB_4^{n+1}, g)$ is a stable minimal hypersurface with  
   \begin{align}
    \|g - g_{Euc}\|_{C^4}\leq \delta_1,&\  & \mbfM(\Sigma)\leq \Lambda, &\  &  \mbfF(|\Sigma|, \scM^{\geq 2}) \geq \epsilon.    \label{Equ_g close, Mass bd, uniform multi 1}
   \end{align}
   Then for every $r\in (0, 1/4)$ and $x\in \spt(\Sigma)\cap \BB_1$, we have \[
    \mbfF_{\BB_4}(|\eta_{x, r^{-1}}(\Sigma)|, \scM^{\geq 2}) \geq \Psi(\epsilon| n, \Lambda).   \]
  \end{Prop}
  Here, $\Psi(\epsilon| n, \Lambda)$ denotes a function in $\epsilon, n, \Lambda$, which tends to $0$ when $n, \Lambda$ are fixed and $\epsilon\to 0$.
  
  To prove Proposition \ref{Prop_Multi 1 in smaller scale}, we need the following preparations.
  \begin{Lem} \label{Lem_Multi 1 cone away from multi>1 vfds}
   For every $\Lambda>0$, there exists $\delta_2(n, \Lambda)>0$ such that for every stable minimal hypercone $C\subset \RR^{n+1}$ with $\|C\|(\BB_1)\leq \Lambda$, we have \[
     \mbfF_{\BB_4}(|C|, \scM^{\geq 2}) \geq 10\delta_2.   \]
  \end{Lem}
  \begin{proof}
   Suppose otherwise, there exist stable minimal hypercones $C_j$ with uniformly bounded density at origin such that $\mbfF_{\BB_4}(|C_j|, \scM^{\geq 2})\to 0$.  Then by \cite{SchoenSimon81} and definition of $\scM^{\geq 2}$ we have, $|C_j|\to m|C_\infty|$ for some stable minimal hypercone $C_\infty$ and integer $m\geq 2$.  Also by the same argument as \cite[Section 4, Claim 5 \& 6]{Sharp17} we see that $C_j$ induces a positive homogeneous degree $1$ Jacobi field $u\in C_{loc}^\infty(C_\infty)$.  

   The goal next is to show that such Jacobi field do not exist.
   Let $\Gamma:= C_\infty \cap \SSp^n$, and since $u$ is homogeneous degree $1$, we can write $u(r\omega) = r\cdot v(\omega)$, $r>0$, $\omega\in \Gamma$.  By rewriting the Jacobi field equation in polar coordinates we derive 
   \begin{align*}
    0 = (\Delta_{C_\infty} + |A_{C_\infty}|^2 )u & = [\partial_r^2 + \frac{n-1}{r}\partial_r + \frac{1}{r^2}(\Delta_{\Gamma} + |A_{\Gamma}|^2)] (r\cdot v) \\
    & = r^{-1}\cdot[\Delta_{\Gamma} v + (|A_{\Gamma}|^2+n-1)v ],
   \end{align*}
   where $\Delta_{\Gamma} v + |A_{\Gamma}|^2 + n-1$ is the Jacobi operator of $\Gamma\subset \SSp^n$.   Since $v>0$, this implies $\Gamma \subset \SSp^n$ is a stable minimal hypersurface, which is impossible.
  \end{proof}
  
  The following Almost Cone Rigidity Lemma follows directly from Schoen-Simon's Compactness Theorem \cite[Theorem 2]{SchoenSimon81} and monotonicity of area.
  \begin{Lem}[Almost Cone Rigidity] \label{Lem_Cone Rigidity}
   For every $\epsilon>0$ and $\Lambda>0$, there exists $\delta_3(\epsilon, \Lambda, n)\in (0, 1)$ such that if $\Sigma \subset (\BB_4^{n+1}, g)$ is a stable minimal hypersurface with $\Sigma\cap \BB_1 \neq \emptyset$ and
   \begin{align*}
    \|g - g_{Euc}\|_{C^4}\leq \delta_3,&\  & \|\Sigma\|(\BB_4) \leq \Lambda, &\  & \theta_g(\mathbf{0}, 4; \|\Sigma\|) - \theta_g(\mathbf{0}, 1; \|\Sigma\|) \leq \delta_3.    
   \end{align*}
   Then there exists a stable minimal hypercone $C\subset \RR^{n+1}$ and $m\in \NN$ such that $\mbfF_{\BB_2}(|\Sigma|, m|C|)\leq \epsilon$.
  \end{Lem}
  
  \begin{proof}[Proof of Proposition \ref{Prop_Multi 1 in smaller scale}]
   Take $\delta_1 = \delta_3(\delta_2, n, \Lambda)/10$, given by Lemma \ref{Lem_Multi 1 cone away from multi>1 vfds} and \ref{Lem_Cone Rigidity}. 
   
   Suppose for contradiction, there exists some $\epsilon>0$ and stable minimal hypersurfaces $\Sigma_j\subset (\BB_4, g_j)$, $x_j\in \spt(\Sigma_j)\cap \BB_1$ and $r_j \in (0, 1/4)$ such that $\|g_j - g\|_{C^4}\leq \delta_1$, $\mbfF_{\BB_4}(|\Sigma_j|, \scM^{\geq 2}) \geq \epsilon$ but $\mbfF_{\BB_4}(|\eta_{x_j, r_j^{-1}}(\Sigma_j)|, \scM^{\geq 2})\to 0$ when $j\to \infty$.  By unique continuation for stable minimal hypersurfaces, $r_j\to 0$ as $j\to \infty$. By possibly replace $\epsilon$ by a smaller $\epsilon'$ and translate-rescale $\Sigma$, WLOG $x_j \equiv \mathbf{0}$.
   
   For each $j\geq 1$, let \[
    I_j:= \{\alpha\in \NN: \theta_{g_j}(\mathbf{0}, 4^{-\alpha}; \|\Sigma_j\|) - \theta_{g_j}(\mathbf{0}, 4^{-\alpha-1}; \|\Sigma_j\|) \geq \delta_1 \}\cup\{+\infty, 0\}   \]
   Since $\theta_{g_j}(\mathbf{0}, r; \|\Sigma\|)$ is non-decreasing in $r$, by the mass upper bound on $\Sigma_j$, we have $\sharp I_j \leq C(\Lambda, n)$. Write 
   \begin{align*}
    I_j=:\{0=\alpha_j^{(0)}< \alpha_j^{(1)}<...<\alpha_j^{(\sharp I_j -1)} = +\infty \}; &\ & \alpha_j^k := +\infty, \ \text{ for }k\geq \sharp I_j.      
   \end{align*}

   Consider 
   \begin{align*}
    i_0:= \sup\Big\{i\in \ZZ_{\geq 0}: \text{ for every sequence }& \{0<s_j\in [4^{-\alpha_j^{(i)}-1}, 1/2)\}_{j\geq 1},\\
    & \liminf_{j\to \infty} \mbfF_{\BB_4}(|\eta_{\mathbf{0}, s_j^{-1}}(\Sigma_j)|, \scM^{\geq 2}) > 0 \Big\}.     
   \end{align*}
   By our contradiction assumption, $0\leq i_0 < C(\Lambda, n)$ and for infinitely many $j$, $\alpha_j^{i_0}<+\infty$.  Also by passing to a subsequence in $j$, suppose $0< s_j\in [4^{-\alpha_j^{(i_0+1)}-1}, 4^{-\alpha_j^{(i_0)}-1})$ such that $\mbfF_{\BB_4}(|\eta_{s_j^{-1}}(\Sigma_j)|, \scM^{\geq 2}) \to 0$ when $j\to \infty$.
   
   Denote for simplicity $\hat{\Sigma}_j:= \eta_{4^{\alpha_j^{(i_0)}+1}}(\Sigma_j)$, by definition of $i_0$ and Schoen-Simon's Compactness Theorem \cite{SchoenSimon81}, we have 
   \begin{align}
    |\hat{\Sigma}_j| \to |\hat{\Sigma}_\infty| \ \ \ \text{ in } \BB_2, \label{Equ_hat(Sigma_j) converg multi 1}
   \end{align}
   for some multiplicity $1$ stable minimal hypersurface $\hat{\Sigma}_\infty$ in $\BB_2$; Also, $\hat{s}_j:= s_j\cdot 4^{\alpha_j^{(i_0)}+1} \in (0, 1)$ satisfies 
   \begin{align}
    \mbfF(|\eta_{\hat{s}_j^{-1}}(\hat{\Sigma}_j)|, \scM^{\geq 2}) \to 0\ \ \  \text{ as } j\to \infty. \label{Equ_hat(s_j)^-1 hat(Sigma_j) to multi 2}
   \end{align}
   And by definition of $I_j$, for every integer $l\in [0, -\log_4(\hat{s}_j)-2)$, we have 
   \begin{align}
    \theta_{g_j}(\mathbf{0}, 4^{-l}; \|\hat{\Sigma}_j\|) - \theta_{g_j}(\mathbf{0}, 4^{-l-1}; \|\hat{\Sigma}_j\|) \leq \delta_1. \label{Equ_theta diff small for hat(Sigma_j)}
   \end{align}
   By (\ref{Equ_hat(Sigma_j) converg multi 1}) and (\ref{Equ_hat(s_j)^-1 hat(Sigma_j) to multi 2}), $\hat{s}_j \searrow 0$; By (\ref{Equ_theta diff small for hat(Sigma_j)}) and Lemma \ref{Lem_Cone Rigidity}, for every $j\geq 1$ and $l\in [1, -\log_4(\hat{s}_j) -1)$, there exists some stable minimal hypercone $C_j^{(l)}$ and integer $m_j^{(l)} \in \NN$ such that 
   \begin{align}
    \mbfF_{\BB_4}(|\eta_{4^l}(\hat{\Sigma}_j)|, m^{(l)}_j|C_j^{(l)}|)\leq \delta_2. \label{Equ_4^l hat(Sigma_j) close to cone w multip}   
   \end{align}
   
   
   On the other hand, recall by \cite{Ilmanen96}, each tangent cone of a stable minimal hypersurface is a multiplicity $1$ stable hypercone. Hence, by Lemma \ref{Lem_Multi 1 cone away from multi>1 vfds}, there exists some integer $l_0 = l_0(\hat{\Sigma}_\infty)>>1$ such that \[
    \mbfF_{\BB_4}(|\eta_{4^{l_0}}(\hat{\Sigma}_\infty)|, \scM^{\geq 2}) \geq 9\delta_2.  \]
   And since $|\hat{\Sigma}_j|\to |\hat{\Sigma}_\infty|$, we can choose $j_0>>1$ such that for every $j\geq j_0$, $s_j < 4^{-l_0-3}$ and \[
    \mbfF_{\BB_4}(|\eta_{4^{l_0}}(\hat{\Sigma}_j)|, \scM^{\geq 2}) \geq 8\delta_2   \]
   Combine this with (\ref{Equ_4^l hat(Sigma_j) close to cone w multip}), we derive $\mbfF_{\BB_2}(m^{(l_0)}_j|C^{(l_0)}_j|, \scM^{\geq 2})\geq 7\delta_2$, and hence $m^{(l_0)}_j = 1$.  
   Moreover, recall for every two integral varifold $V_1, V_2 \in \cI\cV_n(\BB_2^{n+1})$ and every $\lambda>1$, we have \[
    \mbfF_{\BB_2}((\eta_\lambda)_{\sharp} V_1, (\eta_\lambda)_{\sharp} V_2) \leq \lambda \cdot\mbfF_{\BB_2}(V_1, V_2).  \]
   Thus combine with (\ref{Equ_4^l hat(Sigma_j) close to cone w multip}) we see, 
   \begin{align*}
    &\ \mbfF_{\BB_2}(m^{(l)}_j|C^{(l)}_j|, m^{(l+1)}_j|C^{(l+1)}_j|)  \\
    \leq &\ \mbfF_{\BB_2}(m^{(l)}_j|C^{(l)}_j|, |\eta_{4^{l+1}}(\hat{\Sigma}_j)|) + \mbfF_{\BB_2}(|\eta_{4^{l+1}}(\hat{\Sigma}_j)|, m^{(l+1)}_j|C^{(l+1)}_j|) \leq 4\delta_2 + \delta_2 = 5\delta_2.     
   \end{align*}
   Therefore, using Lemma \ref{Lem_Multi 1 cone away from multi>1 vfds} and starting at $l_0$, we can prove inductively that $m^{(l)}_j = 1$ for integer $l\in [l_0, -\log_4(\hat{s}_j)-1)$ and again by Lemma \ref{Lem_Multi 1 cone away from multi>1 vfds} and (\ref{Equ_4^l hat(Sigma_j) close to cone w multip}), we have $\limsup_{j\to \infty}\mbfF_{\BB_4}(|\eta_{\hat{s}_j^{-1}}(\hat{\Sigma}_j)|, \scM^{\geq 2}) \geq \delta_2 > 0$. This contradicts to (\ref{Equ_hat(s_j)^-1 hat(Sigma_j) to multi 2}).
  \end{proof}

  Following the argument in \cite{BombieriGiusti72_Harnack}, an immediate corollary of Proposition \ref{Prop_Multi 1 in smaller scale} is the Neumann-Sobolev type inequality below for stable minimal hypersurfaces,
  \begin{Cor}[Neumann-Sobolev Inequality] \label{Cor_Neum-Sobolev Ineq}
   Let $\epsilon, \Lambda>0$; $\delta_1$ be in Proposition \ref{Prop_Multi 1 in smaller scale}.  Then there exists $\beta(n, \epsilon, \Lambda)\in (0, 1/2)$, $C(n, \epsilon, \Lambda)>0$ with the following property.
   If $\Sigma$ is a stable minimal hypersurface in $(\BB_4^{n+1}, g)$ satisfying (\ref{Equ_g close, Mass bd, uniform multi 1}). Then for every $f\in W^{1,1}(\Sigma)$, we have \[
    \inf_{k\in \RR} \|f - k\|_{L^{\frac{n}{n-1}}(\BB_\beta; \|\Sigma\|)} \leq C\cdot \int_{\BB_4} |\nabla_\Sigma f|\ d\|\Sigma\|.   \]
  \end{Cor}
  \begin{proof}
   We shall first show the following,\\
   
\noindent   \textbf{Claim.} There exists $\beta(n, \epsilon, \Lambda)\in (0, 1/2)$, $\tilde{C}(n, \epsilon, \Lambda)>0$ with the following property.
   If $\Sigma$ is a stable minimal hypersurface in $(\BB_4^{n+1}, g)$ satisfying (\ref{Equ_g close, Mass bd, uniform multi 1}); $T_1, T_2 \in \mbfI^n(\BB_4)$ are integral currents such that 
   \begin{align}
   [\Sigma] = T_1 + T_2; &\  & \|\Sigma\| = \|T_1\| + \|T_2\|.  \label{Equ_Sigma decomp T_1+T_2} 
   \end{align}
   Then we have for every $x\in \spt(\Sigma)\cap \BB_1$, 
   \begin{align}
    \mbfM_{\BB_4}(\partial T_1) = \mbfM_{\BB_4}(\partial T_2) \geq \tilde{C}^{-1}\cdot min\{\mbfM_{\BB_\beta(x)}(T_1), \mbfM_{\BB_\beta(x)}(T_2)\}^{1-1/n}.  \label{Equ_Neumann Isop ineq} 
   \end{align}
   \textit{Proof of Claim.}  By Lemma \ref{Lem_Multi 1 cone away from multi>1 vfds}, by a translation and rescaling of $\Sigma$ and considering a smaller $\epsilon$, WLOG $x = \mathbf{0}$. Suppose the Claim fails, then there are stable minimal hypersurfaces $\Sigma^j$ satisfying (\ref{Equ_g close, Mass bd, uniform multi 1}) and integral currents $T^j_1, T^j_2 \in \mbfI_n(\BB_4)$ such that $[\Sigma^j] = T^j_1 + T^j_2$, $\|\Sigma^j\| = \|T^j_1\| + \|T^j_2\|$ and that
   \begin{align}
    \mbfM_{\BB_4}(\partial T^j_1) = \mbfM_{\BB_4}(\partial T^j_2) < \frac{1}{j}\cdot min\{\mbfM_{\BB_{j^{-1}}}(T^j_1), \mbfM_{\BB_{j^{-1}}}(T^j_2)\}^{1-1/n}.  \label{Equ_Reverse Neumann Isop ineq}
   \end{align}
   Also by slicing Theorem \cite{Simon83_GMT}, for each $T = T^j_i$ ($j\geq 1$, $i=1,2$) and a.e. $s\in (j^{-1}, 4)$, \[
    \mbfM_{\BB_4}(\partial (T \llcorner \BB_s)) \leq \mbfM_{\BB_s}(\partial T) + \frac{d}{ds} \mbfM_{\BB_s}(T).   \]
   Combined with (\ref{Equ_Reverse Neumann Isop ineq}) and Micheal-Simon inequality \cite[Theorem 18.6]{Simon83_GMT}, this implies 
   \begin{align}
    \mbfM_{\BB_s}(T)^{\frac{n-1}{n}} \leq \frac{S(n)}{j}\mbfM_{\BB_s}(T)^{\frac{n-1}{n}} + \frac{d}{ds} \mbfM_{\BB_s}(T),\ \ \ \text{ for a.e. }t\in (j^{-1}, 1).  \label{Equ_M_(B_t)(T) Diff ineq}   
   \end{align}
   Take $j\geq 10S(n)$ in (\ref{Equ_M_(B_t)(T) Diff ineq}) and integrate over $(j^{-1}, t)$, we derive
   \begin{align}
    \mbfM_{\BB_t}(T^j_i) \geq \delta(n)\cdot (t-j^{-1})^n, \ \ \ \forall t\in (j^{-1}, 1).  \label{Equ_Mass lower decay for T^j_i}
   \end{align}
   By Proposition \ref{Prop_Multi 1 in smaller scale} and Schoen-Simon compactness, when $j\to \infty$, $|\eta_{\mathbf{0}, j}(\Sigma^j)|\to |\Sigma^\infty|$ in $\cI\cV_n(\RR^n)$ for some multiplicity $1$ stable minimal hypersurface $\Sigma^\infty$. Moreover, since the tangent cone $C$ of $\Sigma^\infty$ at infinity is also some rescaled limit of $\Sigma^j$, we must have by Proposition \ref{Prop_Multi 1 in smaller scale} that $C$ is of multiplicity $1$ and hence $\Sigma^\infty$ is connected.
   
   By (\ref{Equ_Reverse Neumann Isop ineq}), Fleming-Federer Compactness Theorem and area monotonicity formula for minimal hypersurfaces,  \[
     \mbfF((\eta_{\mathbf{0},j})_\sharp T^j_i, T^\infty_i) \to 0,   \] 
   for some $T^\infty_i \in \mbfI_n(\RR^{n+1})$ ($i=1,2$) such that
   \begin{align*}
    \partial T^\infty_1 = \partial T^\infty_2 = 0; &\ & T^\infty_1 + T^\infty_2 = [\Sigma^\infty], &\ & \|T^\infty_1\| + \|T^\infty_2\| = \|\Sigma^\infty\|.
   \end{align*}
   And by (\ref{Equ_Mass lower decay for T^j_i}), $\spt(T^\infty_1)\cap \BB_2 \neq \emptyset$, $\spt(T^\infty_2)\cap \BB_2 \neq \emptyset$. But by constancy Theorem \cite[Theorem 26.27]{Simon83_GMT}, this is impossible for a connected stable minimal hypersurface $\Sigma^\infty$.  Hence the Claim is proved.\\
   
   Now to prove Corollary \ref{Cor_Neum-Sobolev Ineq}, suppose WLOG $f\in C^1(\Sigma)$, and denote for simplicity $p:= n/(n-1)$. Choose $k\in \RR$ such that 
   \begin{align}
    \mbfM_{\BB_\beta}(\{f>k\}),\ \mbfM_{\BB_\beta}(\{f<k\}) \leq \frac{1}{2} \mbfM_{\BB_\beta}(\Sigma).  \label{Equ_Choice of k wrt f}
   \end{align}
   Let $f^\pm:= (f-k)^\pm = max\{\pm(f-k), 0\}$; $A^\pm(t):= [\Sigma]\llcorner \{f^\pm > t\}$.  Then by co-area formula, we have
   \begin{align*}
    \int_{\BB_4\cap \Sigma} |\nabla_\Sigma f^+|\ d\|\Sigma\| & = \int_0^\infty \mbfM_{\BB_4}(\partial A^+(t))\ dt \\
    & \geq \tilde{C}^{-1}\int_0^\infty \mbfM_{\BB_\beta}(A^+(t))^{1/p}\ dt \\
    & \geq \tilde{C}^{-1}\Big(p\cdot\int_0^\infty \mbfM_{\BB_\beta}(A^+(t)) t^{p-1}\ dt \Big)^{1/p} \\
    & = C^{-1}\|f^+\|_{L^p(\BB_\beta; \|\Sigma\|)}, 
   \end{align*}
   where the first inequality follows from the Claim and choice of $k$; the second inequality follows from the following Hardy-Littlewood-Polya Inequality: If $p>1$ and $\eta\geq 0$ is a non-increasing function on $(0, +\infty)$, then \[
    \Big(\int_0^\infty \eta^{1/p}\ dt \Big)^p \geq p\int_0^\infty \eta(t)t^{p-1}\ dt.   \]
   Similarly one derive $\int_{\BB_4} |\nabla_\Sigma f^-|\ d\|\Sigma\| \geq C^{-1}\|f^-\|_{L^p(\BB_\beta; \|\Sigma\|)}$.  Hence Corollary \ref{Cor_Neum-Sobolev Ineq} is proved.   
  \end{proof}

  By De Giorgi-Moser iteration process \cite{GilbargTrudinger01} and the abstract John-Nirenberg Ineq \cite[Section 4]{BombieriGiusti72_Harnack}, Corollary \ref{Cor_Neum-Sobolev Ineq} implies the following Harnack inequality for stable minimal hypersurfaces.
  \begin{Cor}[Harnack Inequality] \label{Cor_Harnack Ineq}
   Let $\Sigma$ be in Corollary \ref{Cor_Neum-Sobolev Ineq}. There exists $\eta(\epsilon, n, \Lambda)\in (0, 1)$ such that if $0<u\in W^{1,2}_{loc}(\Sigma)$ satisfies $\Delta_\Sigma u - \Lambda u \leq 0$ on $\Sigma$ in the distribution sense, then for every $p\in (0, \frac{n}{n-2})$, every $x\in \spt(\Sigma)\cap \BB_1$ and every $r\in (0, \eta)$, we have \[
     \big(\frac{1}{\|\Sigma\|(\BB_r(x))}\int_{\Sigma\cap \BB_r(x)} |u|^p\ d\|\Sigma\| \big)^{1/p} \leq C(p, \epsilon, \Lambda, n)\cdot\inf_{\Sigma\cap \BB_r(x)} u.   \]
  \end{Cor}
  \begin{Rem}
   By a slight modification, a similar Neumann-Sobolev Inequality and Harnack inequality still holds for minimal hypersurfaces with bounded index, where the constant will also depend on the index upper bound.
   
   On the other hand, the assumption in (\ref{Equ_g close, Mass bd, uniform multi 1}) that $\Sigma$ is away from varifolds with multiplicity $\geq 2$ at scale $1$  CANNOT be dropped.  A typical counterexample is a sequence of pairs of parallel hyperplanes $V_j := |L^{(1)}_j| + |L^{(2)}_j|$ converging to a multiplicity $2$ hyperplane, and a sequence of functions on $\spt(V_j)$ which is $1$ on $L^{(1)}_j$ and $1/j$ on $L^{(2)}_j$.  Counterexample in unstable situation could be a blow down sequence of catenoids converging to a multiplicity $2$ hyperplane and a positive harmonic function which approaches $1$ on one end and $0$ on the other end.
  \end{Rem}

\section{Growth Rate Lower Bound for Positive Jacobi Fields} \label{Sec_Growth Rate Lower Bd}
  The goal of this section is to prove the following.
  \begin{Prop} \label{Prop_Growth Rate Lower Bd}
   Let $\epsilon \in (0, 1)$, $\Lambda>0$, $n\geq 7$; \[
    \gamma_n := -\frac{n-2}{2} + \sqrt{(\frac{n-2}{2})^2 - n+1}   \] 
   For every $\gamma>\gamma_n$, there exists $C = C(\epsilon, \Lambda, n, \gamma)>0$ and $\tau_0(\epsilon, \Lambda, n, \gamma)\in (0, 1/2)$ such that the following hold.
   
   Let $\Sigma \subset (\BB_4^{n+1}, g)$ be a two-sided stable minimal hypersurface satisfying (\ref{Equ_g close, Mass bd, uniform multi 1});
   $u\in W^{1,2}_{loc}(\Sigma)$ be a positive super-solution of $(\Delta_\Sigma + |A_\Sigma|^2 - \Lambda)u = 0$, i.e.   \[
    \int_\Sigma \nabla_\Sigma u \cdot \nabla_\Sigma \xi + (\Lambda-|A_\Sigma|^2)u\xi\ \geq 0,\ \ \ \forall \xi\in C^1_c(\Sigma, \RR_{\geq 0}).   \]
   Then for every $p\in \Sing(\Sigma)\cap \BB_1$ and every $r\in (0, \tau_0)$, we have  
   \begin{align}
    \inf_{\Sigma\cap \BB_r(p)} u \geq C\cdot \|u\|_{L^1(\Sigma \cap \BB_{\tau_0}(p))}\cdot r^\gamma.
   \end{align}
  \end{Prop}
  Note that the exponent $\gamma_n$ is optimal and can be realized by a homogeneous positive Jacobi field on Simons cone.
  
  The key for the proof is the following Growth Lemma.  Let \[
   \scC_n^{=1}(\Lambda):= \{|C|: C\subset \RR^{n+1} \text{ stable minimal hypercone with }\|C\|(\BB_4)\leq \Lambda\}\setminus\{\text{hyperplanes}\},  \]
  be the space of multiplicity $1$ varifolds associated to nontrivial stable minimal hypercones.   For a function $\phi$ over $\Sigma$, $p>0$ and a domain $\Omega \subset \Sigma$,  let \[
   \|\phi\|_{L^p_*(\Omega; \|\Sigma\|)} := \big(\|\Sigma\|(\Omega)^{-1}\cdot \int_\Omega |\phi|^p\ d\|\Sigma\| \big)^{1/p}.  \]
  \begin{Lem}[Growth Lemma] \label{Lem_Growth Jac Field}
   Let $\Lambda>0$, $\gamma>\gamma_n$ be fixed as in Prop \ref{Prop_Growth Rate Lower Bd}.   There exists $\delta(n, \Lambda, \gamma)\in (0, 1)$ and $r_0(n, \Lambda, \gamma)\in (0, 1/2)$ with the following property.
   
   Let $\Sigma\subset (\BB_{4r_0^{-1}}, g)$ be a two-sided stable minimal hypersurface with normal field $\nu$, satisfying
   \begin{align}
    \Sing(\Sigma)\ni \mathbf{0}, &\  & \|g-g_{Euc}\|_{C^4}\leq \delta, &\ & \mbfF_{\BB_4}(|\Sigma|, \scC_n^{=1}(\Lambda)) \leq \delta.  \label{Equ_Almost Cone Assumption}
   \end{align} 
   Let $0< u\in W^{1,2}_{loc}(\Sigma)$ be a positive supersolution of $\Delta_\Sigma u + (|A_\Sigma|^2-\delta) u = 0$ in the distribution sense on $\Sigma$, in other words, 
   \begin{align} 
    \int \nabla_\Sigma u \cdot \nabla_\Sigma \xi + (\delta -|A_\Sigma|^2)u\xi\ d\|\Sigma\|\geq 0,\ \ \ \forall \xi\in C^1_c(\Sigma, \RR_{\geq 0}).  \label{Equ_u satisf Perturb Jac supersol}
   \end{align}
   Then $u\in L^1(\BB_1; \|\Sigma\|)$ and \[ 
    \|u\|_{L^1_*(\BB_r; \|\Sigma\|)} \geq \|u\|_{L^1_*(\BB_1; \|\Sigma\|)}\cdot r^\gamma, \ \ \ \forall r\in [r_0/2, r_0].  \]
  \end{Lem}
  In \cite[Theorem 1]{Simon08}, a general Growth Theorem is established for multiplicity $1$ class of minimal submanifolds.  It follows directly from Proposition \ref{Prop_Multi 1 in smaller scale} that minimal hypersurfaces in Euclidean balls together with their blow-up limits forms a multiplicity $1$ class of hypersurfaces. Hence Simon's result applies here to obtain Lemma \ref{Lem_Growth Jac Field}.   
  For sake of completeness, we include in the Appendix \ref{Sec_Growth Lem} a direct proof using Harnack Inequality in Corollary \ref{Cor_Harnack Ineq}.
  
  \begin{proof}[Proof of Proposition \ref{Prop_Growth Rate Lower Bd}]
   WLOG $p=\mathbf{0}$.
   Let $\tau:= min\{r_0^{N+1}: r_0^N \geq \eta(\epsilon, \Lambda, n)/4\}$ be given by Lemma \ref{Lem_Growth Jac Field} and Corollary \ref{Cor_Harnack Ineq}; $\tau_0:= \tau^2$. By definition, $\tau, \tau_0$ only depend on $\epsilon, \Lambda, n$.  Let \[
    J:= \{j\in \ZZ_{\geq 2}: \eta_{\tau^j}^{-1}(\Sigma)\cap \BB_{4r_0^{-1}} \subset (\BB_{4r_0^{-1}}, \tau^{-2j}\eta_{\tau^j}^* g) \text{ satisfies }(\ref{Equ_Almost Cone Assumption}), u\circ \eta_{\tau^j} \text{ satisfies }(\ref{Equ_u satisf Perturb Jac supersol}) \}.   \]
   By area monotonicity, Lemma \ref{Lem_Cone Rigidity} and Proposition \ref{Prop_Multi 1 in smaller scale}, we see that 
   \begin{align}
    \sharp (\NN \setminus J) \leq N(n, \Lambda, \epsilon, \gamma)<+\infty.  \label{Equ_Bad Scale Uniformly Bounded}   
   \end{align}
   Now for every $r\in (0, \tau_0)$, let $l\in \ZZ_{\geq 2}$ be such that $\tau^{l+1}\leq r < \tau^l$.  Denote for simplicity $I_u(r):= \|u\|_{L^1_*(\BB_r; \|\Sigma\|)}$.  Note that by Corollary \ref{Cor_Harnack Ineq}, for every $j\geq 1$,
   \begin{align}
    I_u(\tau^{j+1}) \geq \frac{1}{\|\Sigma\|(\BB_{\tau^{j+1}})} \inf_{\BB_{\tau^j}} u \geq c(n, \Lambda, \epsilon)I_u(\tau^j).  \label{Equ_Rough Est for I_u(tau^j) btwn Scales}
   \end{align}
   And for every $j\in J$, by Lemma \ref{Lem_Growth Jac Field}, we have better estimate,
   \begin{align}
    I_u(\tau^{j+1}) \geq I_u(\tau^j) \cdot \tau^\gamma .  \label{Equ_Refined Est for I_u(tau^j) btwn Scales}
   \end{align}
   Combine Corollary \ref{Cor_Harnack Ineq}, (\ref{Equ_Bad Scale Uniformly Bounded}), (\ref{Equ_Rough Est for I_u(tau^j) btwn Scales}) and (\ref{Equ_Refined Est for I_u(tau^j) btwn Scales}) we have,
   \begin{align*}
    \inf_{\Sigma\cap \BB_r} u & \geq c(n, \Lambda, \epsilon) \|u\|_{L^1_*(\BB_{\tau^l; \|\Sigma\|})} \\
    & = c(n, \Lambda, \epsilon)\Big(\prod_{j=2}^{l-1} \frac{I_u(\tau^{j+1})}{I_u(\tau^j)} \Big) \cdot \|u\|_{L^1_*(\BB_{\tau_0}; \|\Sigma\|)} \\
    & = c(n, \Lambda, \epsilon)\Big(\prod_{2\leq j\leq l-1; j\in J} \frac{I_u(\tau^{j+1})}{I_u(\tau^j)} \Big)\cdot \Big(\prod_{2\leq j\leq l-1; j\notin J} \frac{I_u(\tau^{j+1})}{I_u(\tau^j)} \Big) \cdot \|u\|_{L^1_*(\BB_{\tau_0}; \|\Sigma\|)} \\
    & \geq c(n, \Lambda, \epsilon)\cdot(\tau^\gamma)^{l-\sharp (\NN\setminus J)}\cdot c(n, \Lambda, \epsilon)^{\sharp (\NN\setminus J)} \cdot \|u\|_{L^1_*(\BB_{\tau_0}; \|\Sigma\|)} \\
    & \geq c(n, \Lambda, \epsilon, \gamma)\|u\|_{L^1_*(\BB_{\tau_0}; \|\Sigma\|)} \cdot r^\gamma .
   \end{align*}
   This finishes the proof of the Proposition.
  \end{proof}
  
  Since $\gamma_n< -1 <0$, a direct corollary of Proposition \ref{Prop_Growth Rate Lower Bd} is,
  \begin{Cor} \label{Cor_pos super Jac field div near sing}
   Let $(M, g)$ be a smooth Riemannian manifold (not necessarily complete), $\Lambda>0$.  Let $\Sigma \subset M$ be a two-sided stable minimal hypersurface with optimal regularity.  Let $u \in C^2_{loc}(\Sigma)$ be a positive function such that \[
     \Delta_\Sigma u + |A_\Sigma|^2 u - \Lambda u \leq 0,   \]
   on $\Sigma$.  Then for every $x\in \Sing(\Sigma)$, \[
     \liminf_{y\to x} u(y) = +\infty.   \]
  \end{Cor}

\section{Proof of Main Theorem } \label{Sec_Pf Main Thm}
  Throughout this section, let $E\subset \RR^{n+1}$ be a closed viscosity mean convex cone with connected nonempty interior. Let $\cE := E\cap \SSp^n$ be a closed viscosity mean convex domain in $\SSp^n$.
  Let's start with the following regularity Lemma.
  \begin{Lem} \label{Lem_Regular Minzer w Mean Convex Cone Bdy}
   Suppose $\partial [E]$ is not minimizing in $E$ (This is automatically true if $E$ is not a Caccioppoli set).  Then there exists a closed subset $P\subset E$ such that,
   \begin{enumerate} [(i)]
   \item $P\supset E\setminus \BB_1$ and $P\cap \BB_1 \subset \Int(E)$;
   \item $\partial P \cap \BB_1$ is a smooth radial graph over $\Int(E)\cap \SSp^n$, i.e., there exists $v\in C^\infty(\Int(\cE); (0, 1))$ such that \[
     P\cap \BB_1 = \{r\omega: \omega\in \Int(\cE), v(\omega)\leq r < 1\};   \] 
   \item $\partial [P]\llcorner \BB_1$ is area-minimizing in $\BB_1\cap \Int(E)$.
   \end{enumerate}
  \end{Lem}
  \begin{proof}
   Let's first further assume that $\partial E\cap \SSp^n$ is $C^{1,1}$ optimally regular mean convex in $\SSp^n$. Consider minimize perimeter in $\BB_2$ among family of Caccioppoli sets \[
     \{Q: E\setminus \BB_1\subset Q \subset E\},   \]
   to get a Caccioppoli set $P\subset E$. Since $\partial [E]$ is not minimizing in $E$, we must have $P\neq E$.  By \cite[Chapter 2]{Lin85}, since $E$ and $\Clos(\BB_1)$ are $C^{1, 1}$ mean convex away from a low dimensional subset, we know that $\|\partial [P]\|$ is a stationary integral varifold in $(\RR^{n+1}\setminus \partial E )\cup \BB_1$.  Then by Solomon-White's strong maximum principle \cite{SolomonWhite89_Maxim} and Ilmanen's strong-maximum principle \cite{Ilmanen96}, either $E = P$ (which is impossible) or $\spt(P)\cap \partial E\cap \BB_1 \subset \{\mathbf{0}\}$. 
   If $\mathbf{0}\in \spt(P)$, then by Lemma \ref{Lem_Cptness of perimeter minzer in E_j}, the tangent cone of $P$ at $\mathbf{0}$ is also perimeter-minimizing in $E$, hence cannot coincide with $E$, this also violates the strong maximum principle for cross sections of $E$ and $P$ in $SS^n$.
   
   Now we have shown that $\spt(P)\cap \BB_1 \subset \Int(E)$, since $P$ minimize perimeter in $E\cap \BB_1$, we have $\dim \Sing(\partial [P]) \leq n-7$.  In particular, WLOG $P = \spt[P]$. Consider the rescalings of $P$ and 
   \begin{align}
     r:= \sup\{s>0 : \eta_t(P)\supset P,  \forall t\in (0, s)\} \in (0, 1].  \label{Equ_First touching scale}   
   \end{align}
   If $r<1$, then there exists $y\in \BB_1\cap \partial P\cap \eta_r(\partial P)$ being a touching point of $P\subset \eta_r(P)$ on the boundary, which violates strong maximum principle \cite{Ilmanen96}.  Hence $r=1$.
   
   Let $\nu_P$ be the inward pointed normal field defined on $\Reg(\partial P)$. Consider the Jacobi field $\phi:= X\cdot \nu_P$ on $\Reg(\partial P)$ induced by rescalings, where $X$ is the position vector.  Since $r=1$ in (\ref{Equ_First touching scale}), we know that $\phi\geq 0$ on $\Reg(\partial P\cap \BB_1)$.  And since $\phi$ satisfies the Jacobi field equation \[
     \Delta \phi + |A_{\partial P}|^2 \phi = 0,    \]
   by strong maximum principle, either $\phi \equiv 0$ (which then implies $P$ being a cone, contradicts to that $\mathbf{0}\notin P$), or $\phi>0$ on regular part of $\partial P$.  Then by Corollary \ref{Cor_pos super Jac field div near sing} and that $\phi\leq 1$, we know that $\Sing(\partial P\cap \BB_1) = \emptyset$.  
   And thus, by $\phi>0$ and $r=1$ in (\ref{Equ_First touching scale}), we know that the projection $x\mapsto x/|x|$ is injective restricting to $\partial P\cap \BB_1$ and onto $\cE \subset \SSp^n$, which indicates that $\partial P\cap \BB_1$ is a radial graph.
   
   Now for a general $E$, by Corollary \ref{Cor_C^1,1 optimal reg mean convex approx}, we can approximate $\cE$ from interior by a family of connected $C^{1,1}$-optimally regular mean convex closed subset $\cE_j\subset \SSp^n$.  Let $P_j\subset C(\cE_j)$ be the corresponding closed subset given above, which is perimeter minimizing in $C(\cE_j)\cap \BB_1$ and has stationary boundary in $(\RR^{n+1}\setminus \partial C(\cE_j)) \cup \BB_1$, with inward pointed normal vector $\nu_j$ on $\partial P_j \cap \BB_1$, and $\phi_j:= r\partial_r \cdot \nu_j$ be the Jacobi field on $\partial P_j \cap \BB_1$ induced by rescaling.  
   Let $j\to \infty$, by Lemma \ref{Lem_Cptness of perimeter minzer in E_j}, after passing to a subsequence, $P_j$ converges to some closed subset $P \subset E$ locally in the Hausdorff distant sense, where $P\supset E\setminus \BB_1$ and $|\partial P|$ is a stable minimal hypersurface in $(\RR^{n+1}\setminus \partial E)\cup \BB_1$.  Thus by mean convexity of $\Clos(\BB_1)$, we have $\Clos(P\cap \BB_1) \cap \SSp^n = \cE$.  
   Also, on $\Reg(\partial P\cap \BB_1)$, $\phi_j$ subconverges to $\phi = X \cdot \nu_{\partial P}\geq 0$, which is a Jacobi field on $\partial P\cap \BB_1$.  By strong maximum principle, either $\phi \equiv 0$ (which implies that $P$ is a cone, and then $E = P$ with boundary being a stable minimal hypercone, reduces to the previous situation), or $\phi >0$ everywhere.  Repeat the argument above using Corollary \ref{Cor_pos super Jac field div near sing}, we know that $P$ is also a smooth radial graph over $\cE\subset \SSp^n$. 
  \end{proof}
  
  Now we are ready to prove the following stronger version of Theorem \ref{Thm_Intro_MinSmooth}.
  \begin{Thm} \label{Thm_Pf_MinSmooth}
   Let $E\subset \RR^{n+1}$ be closed and invariant under scaling. Suppose $C = \partial [E]$ is a stationary integral cone and is area-minimizing in $E$. Then there exists $S\subset \Int(E)$ minimizing in $E$ and foliates $\Int(E)$ by rescaling.
  \end{Thm}
  \begin{proof}
   By Corollary \ref{Cor_C^1,1 optimal reg mean convex approx}, there exists a increasing sequence of connected closed $C^{1,1}$ optimally regular mean convex domain $\cE_j \subset \Int(\cE)$ approximating $\cE$ such that $\limsup_j\|\partial \cE_j\|(\SSp^n) = \|\partial \cE\|(\SSp^n)$.  And since $\cE_j \subset \Int(\cE_{j+1})$, by Frankel property of minimal hypersurfaces in $\SSp^n$, for $j>>1$, $E_j:= C(\cE_j)$ do not have stationary boundary, hence are not perimeter minimizing in $E_j$.  
   By Lemma \ref{Lem_Regular Minzer w Mean Convex Cone Bdy}, there exist closed subsets $P_j \subset C(\cE_j)$ such that $P_j\cap \BB_1$ are smooth minimizing radial graphs in $\Int(C(\cE_j))\cap \BB_1$. When $j\to \infty$, by Lemma \ref{Lem_Cptness of perimeter minzer in E_j}, $P_j$ $\mbfF$-subconverges to some Caccioppoli set $P_\infty \subset E$ with $\partial P_\infty$ a stable minimal hypersurface in $(\RR^{n+1}\setminus \partial E)\cup \BB_1$ and $P_\infty$ coincide with $E$ on $\BB_1^c$. Moreover, as a minimizing boundary in $C(\cE_j)$, $\partial [P_j]$ satisfies \[
    \|\partial P_j\|(\BB_1)\leq \|\partial [E_j]\|(\BB_1).  \]
   Hence by taking $j\to \infty$ we have,
   \begin{align}
    \|\partial [P_\infty]\|(\BB_1) \leq \|\partial [E]\|(\BB_1).  \label{Equ_P_infty minz perim in E}   
   \end{align}
   Since $E$ is perimeter minimizing in $E$, (\ref{Equ_P_infty minz perim in E}) implies that $P_\infty$ also minimizes perimeter in $\BB_2\cap E$. Hence by unique continuation of stable minimal hypersurfaces, $P_\infty = E$.  Therefore, when $j\to \infty$,  \[
     d_j:= dist(\mathbf{0}, P_j) \searrow 0.   \]
   
   By Lemma \ref{Lem_Cptness of perimeter minzer in E_j}, when $j\to \infty$, $\hat{P}_j:= \eta_{1/d_j}(P_j)$ are perimeter-minimizing in $\BB_{1/d_j}\cap E_j$ and hence $\mbfF$-subconverges to some $\hat{P}_\infty \subset E$, minimizing perimeter in $\Int(E)$ and has $\partial \hat{P}_\infty$ a stable minimal hypersurface in $\RR^{n+1}$.  Since \[
    dsit(\mathbf{0}, \spt(\hat{P}_\infty)) = \lim_{j\to \infty} dist(\mathbf{0},\hat{P}_j) = 1,  \]
   we know that $\partial \hat{P}_\infty$ is not supported on $\partial E$. Hence by strong maximum principle \cite{SolomonWhite89_Maxim, Ilmanen96}, $\spt(\hat{P}_\infty) \subset \Int(E)$ and $\hat{P}_\infty$ also minimize perimeter in $E$. 
   
   Let $\hat{\nu}_j$ be the inward pointed normal field of $\hat{P}_j$ on $\partial \hat{P}_j\cap \BB_{1/d_j}$, $1\leq j\leq \infty$; $\psi_j:= X\cdot \hat{\nu}_j$ be the Jacobi field induced by rescaling defined on $\Reg(\partial \hat{P}_j)\cap \BB_{1/d_j}$.  Since $\hat{P}_j$ are also radial graphs, when $j<\infty$, $\psi_j>0$.
   When $j\to \infty$, $\psi_j$ subconverges to $\psi_\infty$ on $\Reg(\partial \hat{P}_\infty)$.  Hence $\psi_\infty\geq 0$.  Again by strong maximum principle for Jacobi field equation, either $\psi_\infty \equiv 0$ (which implies $\hat{P}_\infty$ is a cone and violates that $dist(\mathbf{0}, \hat{P}_\infty) = 1$), or $\psi_\infty > 0$ everywhere.  Then by Corollary \ref{Cor_pos super Jac field div near sing}, $\Sing(\partial \hat{P}_\infty) = \emptyset$ and $S = \partial \hat{P}_\infty \subset \Int(E)$ is area-minimizing in $E$.  Furthermore, $\psi_\infty >0$ everywhere implies, by the same argument as Lemma \ref{Lem_Regular Minzer w Mean Convex Cone Bdy}, that $S$ is a smooth radial graph over $\cE$, hence foliates $\Int(E)$ by rescalings.  
  \end{proof}

  Now we turn to the non-minimizing case,
  \begin{Thm} \label{Thm_Pf_MeanConv}
   Let $E\subset \RR^{n+1}$ be a closed, viscosity mean convex cone with connected interior. Suppose $C = \partial [E]$ is NOT minimizing in $E$ (This is the case when $E$ is not a Caccioppoli set). 
   
   Let $\cD:=\bigcup_{t\geq 0} D_t\times \{t\} \subset \RR^{n+1}\times \RR$ be the level set flow of $D_0:= \RR^{n+1}\setminus \Int(E)$.  Then $D_t = \eta_{\sqrt{t}}(D_1)$, $\forall t>0$; $D_1$ is a smooth domain and $\partial D_1$ is a self-expander supported in $\Int(E)$ and minimizing $\mbfE$-functional in $\Int(E)$. 
   
   Moreover, $\partial D_1$ is a radial graph over $\cE$, i.e. there exists $v\in C^\infty(\cE;(0, +\infty))$ which diverges near $\partial \cE$ such that\[
     D_1^c = \{r\omega: \omega\in \cE, r>v(\omega)\}.   \]
   In particular, the mean curvature of $\partial D_1$ with respect to outward pointed normal field is positive along $\partial D_1$.
  \end{Thm}
  
  \begin{proof} 
   That $D_t = \eta_{\sqrt{t}}(D_1)$ follows from scaling invariance of $D_0$ and the uniqueness of level set flow, see Proposition \ref{Prop_White_Exist LSF};   
   
   Now let $\cE_j \subset \Int(\cE)$ be the $C^{1,1}$-optimally regular closed connected mean convex approximation given by Corollary \ref{Cor_C^1,1 optimal reg mean convex approx}; Thus $E_j:= C(\cE_j)$ satisfies the assumption in Lemma \ref{Lem_Exist Self-Expander w Obstacle}. 
   Let $D_0^j:= \RR^{n+1}\setminus \Int(E_j)$.  We know that $\Int(D_0^j) \supset D_0 \setminus \{\mathbf{0}\}$ and $\partial D_0^j\to \partial D_0$ locally in Hausdorff distant sense as $j\to \infty$.  
         
   Fix $j\geq 1$, by Lemma \ref{Lem_Exist Self-Expander w Obstacle}, there exists a closed Caccioppoli set $F^j \subset (D^j_0)^c = \Int(E_j)$ minimizing $\mbfE$-functional in $E_j$ and distant asymptotic to $E_j$ near infinity. In particular, $\partial F^j$ is a self-expander in $\RR^{n+1}$ with optimal regularity.  We now Claim that
   \begin{align}
    \eta_\lambda(F^j)\cap \Int(D_1) = \emptyset;\ \ \eta_\lambda(\Int(F^j))\cap D_1 = \emptyset, \ \ \ \forall \lambda\geq 1.  \label{Equ_lambda_>1 S^j disjt from D_1}
   \end{align}
   To see this, recall that by avoidance principle of level set flow, we have, locally in the Hausdorff distant sense, \[
     \lim_{\lambda\searrow 0} D_\lambda = \lim_{\lambda\searrow 0} \eta_\lambda(D_1)  \subset D_0 \subset \Int(D_0^j)\cup \{\mathbf{0}\}.   \]
   Hence there exists $\epsilon_j>0$ and $\lambda_j\in (0, 1)$ such that for every $0<\lambda\leq \lambda_j$ and every $\mu\in (0, 1]$,  
   \begin{align}
     dist(D_{\lambda\mu}\cap \cA, \eta_{\sqrt{\lambda}}(F^j)\cap\cA) \geq \frac{1}{2}dist(D_0\cap \cA, (D_0^j)^c \cap \cA) \geq \epsilon_j,  \label{Equ_dist(D_lambda, sqrt(lambda)F^j)>0}   
   \end{align}
   where $\cA:= \AAa(\mathbf{0}; 1/2, 1)$. In particular, (\ref{Equ_dist(D_lambda, sqrt(lambda)F^j)>0}) implies that for every $\mu\in (0, 1]$, \[
     F^j\cap D_\mu \setminus \BB_{1/\sqrt{\lambda_j}} = \emptyset.   \]
   Hence (\ref{Equ_lambda_>1 S^j disjt from D_1}) follows from Theorem \ref{Thm_White_Avoid} and Lemma \ref{Lem_Trancated self-expander is a weak flow} by comparing $\cD$ with the weak set flow \[
     t\mapsto \eta_{\sqrt{t+\epsilon}}( \spt(\partial [F^j])\cap \Clos(\BB_{2/\sqrt{\lambda_j}})),   \] 
   for every $\epsilon>0$ and then send $\epsilon\searrow 0$.
      
   Now let $j\to \infty$ in (\ref{Equ_lambda_>1 S^j disjt from D_1}), by $\mbfE$-minimizing property of $F^j$ and Corollary \ref{Cor_Two-side barrier for self-expander} and Lemma \ref{Lem_Cptness of perimeter minzer in E_j}, we have $F^j \to F^\infty$ in $\bfF$-metric, where $F^\infty = \spt[F^\infty]\subset E$ is a closed Caccioppoli set distant asymptotic to $E$ near infinity, and has $S^\infty := \partial F^\infty$ a $\mbfE$-stable hypersurface with optimal regularity, minimizing $\mbfE$-functional in $\Int(E)$.  By (\ref{Equ_lambda_>1 S^j disjt from D_1}), we have
   \begin{align}
    \eta_\lambda (F^\infty)\cap \Int(D_1) = \emptyset,\ \eta_\lambda(\Int(F^\infty))\cap D_1 = \emptyset, \ \ \ \forall \lambda\geq 1.   \label{Equ_F^infty disjoint D_1}
   \end{align}
   
   Also, we assert that $F^\infty$ is not a cone. In fact otherwise, $E = F^\infty$ whose boundary is a stable minimal hypercone. In particular, $\mathbf{0}\in F^\infty$.  Hence when $j\to \infty$, $d_j:= dist_{\RR^{n+1}}(\mathbf{0}, F^j) \to 0$, and by Lemma \ref{Lem_Cptness of perimeter minzer in E_j}, $\eta_{1/j}(F^j)$ $\mbfF$-subconverges to some Caccioppoli set $\hat{F}^\infty = \spt[\hat{F}^\infty] \subset E$ with distant $1$ to the origin and an optimally regular stable minimal boundary.  By maximum principle \cite{SolomonWhite89_Maxim, Ilmanen96}, $\hat{F}^\infty \subset \Int(E)$.  But since $E$ is not minimizing, by comparing rescalings of $\hat{F}^\infty$ with the Caccioppoli set constructed in Lemma \ref{Lem_Regular Minzer w Mean Convex Cone Bdy} and using maximum principle \cite{SolomonWhite89_Maxim, Ilmanen96}, we get a contradiction.
   
   On the other hand, by Lemma \ref{Lem_Trancated self-expander is a weak flow}, $t\mapsto \RR^{n+1}\setminus \eta_{\sqrt{t}}(\Int(F^\infty))$ is a weak-set flow from $D_0$. Hence by Proposition \ref{Prop_White_Exist LSF},  we have $D_1\supset \RR^{n+1}\setminus \Int(F^\infty)$. Combined with (\ref{Equ_F^infty disjoint D_1}) when $\lambda = 1$, we have $\partial D_1 = \partial F^\infty$ and $\Int(D_1)= \RR^{n+1}\setminus F^\infty$.
   
   Finally by (\ref{Equ_F^infty disjoint D_1}), we have $\eta_\lambda (F^\infty)\subset F^\infty$ for $\lambda\geq 1$.  Hence $\phi := 2H^{S^\infty} = X\cdot \nu_{S^\infty} \geq 0$ on regular part of $S^\infty$ and is not identically $0$ since $S^\infty$ is not a cone. 
   Recall that by self-expanding equation \cite[(3.9)]{Ding20_MinCone_SelfExpander}, we have $(\cL_{S^\infty} + 1) \phi = 0$, where $\cL_{S^\infty}$ is the Jacobi operator for $\mbfE$-functional defined in (\ref{Equ_Jac oper wrt mbfE-functional}). By (\ref{Equ_Jac oper of mbfE vs Jac oper under hat(g)}), this is also a Jacobi-type equation.  Hence by strong maximum principle, $\phi>0$ and is bounded on regular part of $S^\infty$ in each ball. And therefore by Corollary \ref{Cor_pos super Jac field div near sing}, $\Sing(S^\infty) = \emptyset$ and $\spt(S^\infty)$ is a radial graph over $\Int(\cE)$. 
  \end{proof}
  
  \begin{proof}[Proof of Theorem \ref{Thm_Intro_LawsonConj}]
   Let $E\subset \RR^{n+1}$ be a closed cone bounded by $C$.  Recall by Remark \ref{Rem_Viscosity mean convex} (1), $E$ is viscosity mean convex.  
   If $E$ is not perimeter minimizing in $E$, then Theorem \ref{Thm_Pf_MeanConv} provides entirely smooth, strictly mean convex self-expander $ S\subset \Int(E)$ and distant asymptotic to $E$ near infinity. Thus the blow-down of $S$ restricted to $\BB_1$ will be strictly mean convex hypersurface approximating $C\cap \BB_1$ in the Hausdorff distant sense.
   
   If $E$ minimizes perimeter in $E$, Theorem \ref{Thm_Pf_MinSmooth} provides a smooth minimizing hypersurface $S\subset \Int(E)$ with $C$ to be the tangent cone at infinity.  Blow-down $\eta_\lambda (S)$ ($\lambda\searrow 0$) also converges to $C$ locally in the Hausdorff distant sense, and restricting to $\BB_1$, $\eta_\lambda(S)$ are strictly stable.  Let $\phi_1>0$ be the first eigenfunction of Jacobi operator $L_{\eta_\lambda(S)}$ on the domain $\eta_\lambda(S)\cap \BB_1$, then for sufficiently small $\varepsilon$, \[
     \{x - \varepsilon\phi_1(x)\cdot\nu_{\eta_\lambda(S)}: x\in \eta_\lambda(S)\cap \BB_1\},   \]
   is a properly embedded smooth hypersurface in $\BB_1$ with positive mean curvature, and still approximate $C\cap \BB_1$ in Hausdorff distant sense. 
  \end{proof}

   

\appendix
  \section{Growth Lemma for Positive Jacobi Fields} \label{Sec_Growth Lem}
  The goal of this section is to prove Lemma \ref{Lem_Growth Jac Field}.
  We begin with the following uniform control of first eigenfunction for Jacobi operator of cross section. For every $C\in \scC_n^{=1}(\Lambda)$, let $\Gamma=\Gamma(C):= C\cap \SSp^n$ be the cross section of $C$, and for each domain $\Omega\subset \Gamma$, let \[
   \mu_1(\Omega):= \inf\{\int_\Gamma |\nabla_\Gamma \phi|^2 - |A_\Gamma|^2\phi^2 : \phi\in C_c^1(\Omega), \|\phi\|_{L^2(\Gamma)} = 1\}.   \] 
  Denote for simplicity $\mu_1(C):= \mu_1(\Gamma(C))$.  Recall by \cite{Lawson69_Rigidity_QuadraticMH, Perdomo02_1stEigenV_ClifTorus, ZhuJ18_1stEigenV_JacOper_Sphere}, $\mu_1(C) \leq -(n-1)$, and equality holds if and only if $C$ is a quadratic cone.
  
  For each $x\in \Gamma(C)$, following \cite{CheegerNaber13_QuantStratif}, let \[
   r_\Gamma(x):= \sup\{r\in (0, 1): |A_\Gamma|\leq r^{-1} \text{ on }\BB_r(x)\cap \Gamma  \}   \]
  be the regularity scale. Clearly, $r_\Gamma$ is continuous in $x$ and in smooth convergence of $\Gamma$. For each $\delta>0$, denote $\Omega_\delta(C):= \{x\in \Gamma: r_\Gamma(x)>\delta\}$. 
  
  \begin{Lem} \label{Lem_Lower Bound on Reg Scale}
   There exists $\rho_0(n, \Lambda) > 0$ such that $\mu_1(\Omega_{\rho_0}(C)) \leq -(n-1)$ for every $C\in \scC_n^{=1}(\Lambda)$.
  \end{Lem}
  \begin{proof}
   Suppose otherwise, there are $C_j\in \scC_n^{=1}(\Lambda)$ such that $\mu_1(\Omega_{1/j}(C_j)) > -(n-1)$.  
  
   Suppose $C_j \to C_\infty \in \scC_n^{=1}(\Lambda)$ and write $\Gamma_j := C_j\cap \SSp^n$, $1\leq j\leq \infty$.  Then we must have $\mu_1(C_\infty) = -(n-1)$. Since otherwise by \cite{ZhuJ18_1stEigenV_JacOper_Sphere}, $\mu_1(C_\infty)< -(n-1)$ and thus there exists $\psi \in C_c^1(\Gamma(C_\infty))$ such that \[
    \int_{\Gamma_\infty} |\nabla_{\Gamma_\infty} \psi|^2 - |A_{\Gamma_\infty}|^2\psi^2 < -(n-1)\int_{\Gamma_\infty} \psi^2.   \]
   Since the convergence of $\Gamma_j$ is smooth near $\spt(\psi)$ and $r_{\Gamma_j}$ is continuous in $j$ and $x$ near $\spt(\psi)$, this is a contradiction to $\mu_1(\Omega_{1/j}(C_j)) > -(n-1)$.
  
   Now that $\mu_1(C_\infty) = -(n-1)$. By \cite{ZhuJ18_1stEigenV_JacOper_Sphere}, $C_\infty$ is a quadratic cone, which has smooth cross section; Then by Allard Regularity, for $j>>1$, $\mu_1(\Omega_{1/j}(C_j)) = \mu_1(C_j) \leq -(n-1)$, that's a contradiction.  
  \end{proof}

  \begin{Lem} \label{Lem_Growth Jac field, Cone}
   For every $\Lambda>1$, there exists $C(n, \Lambda)>1$ with the following property.  Let $\delta_2(n, \Lambda)>0$ be in Lemma \ref{Lem_Multi 1 cone away from multi>1 vfds} and $\eta(n,\Lambda) = \eta(\delta_2/2, n, \Lambda)\in (0, 1)$ be in Corollary \ref{Cor_Harnack Ineq};  Also let $\gamma_n$ be in Proposition \ref{Prop_Growth Rate Lower Bd}.
   
   If $C\in \scC_n^{=1}(\Lambda)$ is a stable minimal hypercone; $0<u \in C_{loc}^\infty(C\cap \BB_{4\eta^{-1}})$ is a Jacobi field, i.e. $\Delta_C u + |A_C|^2 u = 0$. Then, \[
    \|u\|_{L^1_*(\BB_r; \|C\|)} \geq C(n, \Lambda)^{-1}r^{\gamma_n}\cdot\|u\|_{L^1_*(\BB_1; \|C\|)}, \ \ \ \text{ for every }r\in (0, 1).   \]
  \end{Lem}
  \begin{proof}
   Let $\Gamma:= C\cap \SSp^n$ be the cross section of $C$ in Lemma \ref{Lem_Growth Jac field, Cone}; $\rho_0$ be in Lemma \ref{Lem_Lower Bound on Reg Scale}. Let $\Omega$ be a smooth domain in $\Gamma$ such that $\Omega_{\rho_0}(C) \subset \Omega \subset \Omega_{\rho_0/2}(C)$, and $0<\varphi_\Omega$ be a first Dirichlet eigenfunction of $-\Delta_\Gamma - |A_\Gamma|^2$, i.e. 
   \begin{align}
   \begin{cases}
    - (\Delta_\Gamma + |A_\Gamma|^2) \varphi_\Omega = \mu_1(\Omega)\varphi_\Omega, &\ \text{ on }\Omega; \\   
    \varphi_\Omega = 0, &\ \text{ on }\partial \Omega \text{ if }\partial \Omega \neq \emptyset. 
   \end{cases}  \label{Equ_varphi_Omega def}
   \end{align}
   with eigenvalue $\mu_1(\Omega) \leq \mu_1(\Omega_{\rho_0}(C)) \leq -(n-1)$ by Lemma \ref{Lem_Lower Bound on Reg Scale}.  By definition of $\Omega$, $|A_\Gamma|\leq 2\rho_0^{-1}$ on $\BB_{\rho_0/2}(\Omega)$. Hence by covering $\Omega$ with ball of radius $\rho_0$ and applying classical elliptic estimates \cite{GilbargTrudinger01} we see
   \begin{align}
    \sup_{\Omega} |\varphi_\Omega| \leq C(n, \Lambda)\|\varphi_\Omega\|_{L^1}   \label{Equ_sup of varphi controled by L^1}
   \end{align}
  
   Now let $U(r):= \int_\Omega \varphi_\Omega (\omega) u(r\omega)\ d\omega$.  Since \[
    0 = (\Delta_C + |A_C|^2 )u = \partial_r^2 u + \frac{n-1}{r}\partial_r u + \frac{1}{r^2}(\Delta_\Gamma + |A_\Gamma|^2)u.  \]
   We then have for $r\in (0, 1]$,
   \begin{align}
   \begin{split}
    U''(r) + \frac{n-1}{r}U'(r) & = -\frac{1}{r^2}\int_\Omega \varphi_\Omega(\omega)(\Delta_\Gamma + |A_\Gamma|^2)u(r\omega) \\
   & = -\frac{1}{r^2} \int_{\Omega} (\Delta_\Gamma + |A_\Gamma|^2)\varphi_\Omega(\omega)\cdot u(r\omega) + \frac{1}{r^2}\int_{\partial \Omega}\partial_\zeta \varphi_\Omega\cdot u \\
   & \leq \frac{\mu_1(\Omega)}{r^2}\int_\Omega \varphi_\Omega(\omega)u(r\omega) \\
   & \leq -\frac{n-1}{r^2}U(r),
   \end{split}  \label{Equ_U'' + (n-1)/r U' < -(n-1)/r^2 U}
   \end{align}
   where the first inequality follows from (\ref{Equ_varphi_Omega def}) as well as positivity of $u$ and the inward normal derivative $-\partial_\zeta \varphi$ once $\partial \Omega \neq \emptyset$; the second inequality follows from $\mu_1(\Omega)\leq -(n-1)$ by Lemma \ref{Lem_Lower Bound on Reg Scale}.  Let 
   \begin{align*}
    \gamma_n = -\frac{n-2}{2} + \sqrt{(\frac{n-2}{2})^2 - (n-1)}, &\ & \beta_n:= 2\sqrt{(\frac{n-2}{2})^2 - (n-1)}.
   \end{align*}
   And set $W(s):= U(s^{-1/\beta_n})\cdot s^{\gamma_n/\beta_n}$, $s\geq 1$, then (\ref{Equ_U'' + (n-1)/r U' < -(n-1)/r^2 U}) is equivalent to $W'' (s) \leq 0$ on $s\geq 1$. Since $u>0$ on $C$ and then $W>0$ on $[1,+\infty)$, we must have $W'\geq 0$ on $[1, +\infty)$.  Thus,
   \begin{align}
    U(r) \geq U(1)\cdot r^{\gamma_n}, \ \ \ \forall r\in (0, 1). \label{Equ_U(r)> U(1)r^gamma_n} 
   \end{align}
   We then have for every $r\in (0, 1)$,
   \begin{align*}
    \|u\|_{L^1(\BB_r; \|C\|)}\sup_\Omega \varphi_\Omega 
   & \geq \int_{\BB_r} \varphi_\Omega(\frac{x}{|x|})u(x)\ d\|C\|(x) \\
   & = \int_0^r s^{n-1} U(s)\ ds \geq \int_0^r s^{n-1+\gamma_n} U(1)\ ds  \\
   & = c_n r^{n+\gamma_n}\cdot \int_\Omega \varphi_\Omega(\omega)u(\omega)\ d\omega \\
   & \geq c_n r^{n+\gamma_n}\cdot \|\varphi_\Omega\|_{L^1} \inf_{\BB_1} u \\
   & \geq C(n, \Lambda)^{-1} r^{n+\gamma_n}\cdot \sup_\Omega \varphi_\Omega \cdot \|u\|_{L^1(\BB_1; \|C\|)},
   \end{align*}
   where the second inequality follows from (\ref{Equ_U(r)> U(1)r^gamma_n}), the last inequality follows from (\ref{Equ_sup of varphi controled by L^1}) and Corollary \ref{Cor_Harnack Ineq}.  This finish the proof. 
  \end{proof}
  
  \begin{Lem} \label{Lem_Solving equ with two-side barrier}
   Let $\Sigma\subset (\BB_4, g)$ be a two-sided minimal hypersurface under metric $g$ such that, the Hessian of distant function square satisfies $\nabla_g^2 ( dist_g(\mathbf{0}, \cdot)^2) \geq 1$ on $\BB_2$. (Note that this is automatically true when $g$ is $C^2$ close to $g_{Euc}$.)
   
   Let $u\in W^{1,2}_{loc}(\Sigma)$ be a positive super-solution of $\Delta_\Sigma u + (|A_\Sigma|^2-\delta) u = 0$ in distribution sense, i.e.
   \begin{align}
    \int \nabla_\Sigma u \cdot \nabla_\Sigma \xi + (\delta -|A_\Sigma|^2)u\xi\ d\|\Sigma\|\geq 0,\ \ \ \forall \xi\in C^1_c(\Sigma, \RR_{\geq 0}), \label{Equ_Perturb Jac field distrib supersol}
   \end{align}
   where $\delta\leq n/2$.  
   Then there exists $v\in C^2_{loc}(\Sigma\cap \BB_1)$ solving $\Delta_\Sigma v + (|A_\Sigma|^2-\delta) v = 0$ on $\Sigma\cap \BB_1$ such that  \[
     \frac{1}{2}\inf_{\BB_1} u \leq v \leq u   \]
  \end{Lem}
  \begin{proof}
   \textbf{Claim.} With a positive $u\in W_{loc}^{1,2}(\Sigma)$ satisfying (\ref{Equ_Perturb Jac field distrib supersol}), we have \[
     \int |\nabla \phi|^2 + (\delta - |A_\Sigma|^2)\phi^2\ d\|\Sigma\| \geq 0, \ \ \ \forall \phi\in C_c^1(\Sigma),   \]
   with equality if and only if $\phi \equiv 0$.\\
   \textit{Proof of the Claim.} For every $\phi\in C_c^1(\Sigma)$, let $\psi:= u^{-1}\phi\in W_0^{1,2}(\Sigma)$.  Take $\xi = u\psi^2 = \phi\psi$ in (\ref{Equ_Perturb Jac field distrib supersol}), we have
   \begin{align*}
    0 \leq \int \nabla u\cdot \nabla(\phi \psi) + (\delta-|A_\Sigma|^2)\phi^2\ d\|\Sigma\| = \int |\nabla \phi|^2 + (\delta-|A_\Sigma|^2)\phi^2 - u^2|\nabla\psi|^2\ d\|\Sigma\| .
   \end{align*}
   This finishes the proof of Claim. \\
   
   
   Now we turn to the proof of Lemma \ref{Lem_Solving equ with two-side barrier}.  By a renormalization, suppose WLOG that $\inf_{\BB_1} u = 1$. 
   Consider for every smooth domain $\Omega\subset \subset \Sigma\cap \BB_1$, minimize \[
     J_\Omega[\phi]:= \int |\nabla \phi|^2 + (-|A_\Sigma|^2 + \delta)\phi^2\ d\|\Sigma\|,   \]
   among $\{\phi: \phi-1\in W_0^{1,2}(\Omega)\}$ to get a minimizer $v_\Omega \in C^2(\Clos(\Omega))$ solving $\Delta_\Sigma v_\Omega + (|A_\Sigma|^2 - \delta)v_\Omega = 0$ in $\Omega$. Since $u\geq 1$ on $\Sigma\cap \BB_1$, we know that $(v_\Omega - u)^+ \in W_0^{1,2}(\Sigma)$ is supported in $\Clos(\Omega)$, thus by taking $\xi = (v_\Omega - u)^+$ in (\ref{Equ_Perturb Jac field distrib supersol}) we have, 
   \begin{align*}
    0 & \geq \int -\nabla u\cdot \nabla (v_\Omega - u)^+ - (\delta - |A_\Sigma|^2)u   (v_\Omega - u)^+ \\
      & = \int \nabla (v_\Omega - u)\cdot \nabla (v_\Omega - u)^+ + (\delta - |A_\Sigma|^2) (v_\Omega - u)\cdot (v_\Omega - u)^+ \\
      & = \int |\nabla \xi|^2 + (\delta - |A_\Sigma|^2)\xi^2 \geq 0,
   \end{align*}
   where the first equality follows by integration by parts and the equation satisfied by $v_\Omega$; the second equality follows from the Claim.  Hence the Claim guarantees that $(v_\Omega - u)^+ = 0$, in other words $v_\Omega\leq u$ in $\Omega$.  On the other hands, let $w:= (1+ dist_g(\cdot, \mathbf{0}))/2 \leq 1$ on $\Sigma\cap \BB_1$,  a simple computation shows that \[
    \Delta_\Sigma w + (|A_\Sigma|^2 - \delta)w \geq 0 \ \text{ on }\Sigma\cap \BB_1.  \]
   This provides a sub-solution to the equation.  Repeat the same process above gives $v_\Omega \geq w \geq 1/2$ on $\Omega$.
   
   Now take $\Omega_j$ be an increasing family of smooth domain approximating $\Sigma\cap \BB_1$, $1/2 \leq v_{\Omega_j}\leq u$ be the solution constructed above relative to $\Omega_j$.  By classical Harnack inequality for Elliptic equations \cite{GilbargTrudinger01}, $v_{\Omega_j}$ subconverges to some solution $v\in C_{loc}^2(\Sigma\cap \BB_1)$ of $\Delta_\Sigma v + (|A_\Sigma|^2 - \delta)v = 0$ which satisfies $1/2 \leq v \leq u$, thus finishes proof of the Lemma.   
  \end{proof}

  \begin{proof}[Proof of Lemma \ref{Lem_Growth Jac Field}]
   Take $r_0 = r_0(n, \Lambda, \gamma) \in (0, \eta(\delta_2, \Lambda, n)/8)$ TBD, where $\delta_2 = \delta_2(n, \Lambda)$ be given by Lemma \ref{Lem_Multi 1 cone away from multi>1 vfds} and $\eta$ is given by the Harnack inequality Corollary \ref{Cor_Harnack Ineq}.
  
   Suppose for contradiction that there exists $\gamma>\gamma_n$ and stable minimal hypersurfaces $\Sigma_j \subset (\BB_4, g_j)$ such that $\mathbf{0}\in \Sing(\Sigma_j)$, $g_j \to g_{Euc}$ in $C^4$ and $\mbfF_{\BB_4r_0^{-1}}(|\Sigma_j|, |C|)\to 0$ for some stable minimal hypercone $C\subset \RR^{n+1}$;  But there are $u_j\in W^{1,2}_{loc}(\Sigma_j)$ weak supersolution to $\Delta_{\Sigma_j} u_j + (|A_{\Sigma_j}|^2 - 1/j)u_j = 0$ and $r_j\in [r_0/2, r_0]$, $r_j\to r_\infty\in [r_0/2, r_0]$ satisfying 
   \begin{align}
    \|u_j\|_{L^1_*(\BB_{r_j}; \|\Sigma\|)} < \|u_j\|_{L^1_*(\BB_1; \|\Sigma\|)}\cdot r_j^\gamma  \label{Equ_Inverse Growth Rate}   
   \end{align}
   Note that by Lemma \ref{Lem_Multi 1 cone away from multi>1 vfds} and $\mbfF_{\BB_{4r_0^{-1}}}(|\Sigma|, |C|)\to 0$, we have for $j>>1$, $\mbfF_{\BB_{4r_0^{-1}}}(|\Sigma_j|, \scM^{\geq 2}) \geq \delta_2(n, 2\Lambda)$.  Hence the Harnack inequality Corollary \ref{Cor_Harnack Ineq} applies for functions on $\Sigma_j$.
      
   Also by Lemma \ref{Lem_Solving equ with two-side barrier}, we can take $v_j\in C^2_{loc}(\Sigma_j\cap \BB_1)$ such that on $\Sigma_j\cap \BB_1$,
   \begin{align}
    0< \frac{1}{2}\inf_{\Sigma_j\cap \BB_1} u_j \leq v_j \leq u_j; &\ &
    \Delta_{\Sigma_j} v_j + (|A_{\Sigma_j}|^2 - 1/j) v_j = 0.   \label{Equ_Lower barrier v_j of contradicting u_j}
   \end{align}
   Thus by applying Harnack inequality Corollary \ref{Cor_Harnack Ineq} on $u_j$ and (\ref{Equ_Inverse Growth Rate}), (\ref{Equ_Lower barrier v_j of contradicting u_j}), we have
   \begin{align}
    \|v_j\|_{L^1_*(\BB_{r_j})} \leq C(n, \Lambda) \inf_{\Sigma_j\cap \BB_1} u_j \cdot r_j^\gamma \leq  C(n, \Lambda) \inf_{\Sigma_j\cap \BB_1} v_j \cdot r_j^\gamma.  \label{Equ_|v_j|_L^1(r_j) < r_j^gamma}
   \end{align}
   
   On the other hand, fix a regular point $x_\infty\in C\cap \BB_{\eta/4}$ and suppose $\Sigma_j\ni x_j \to x_\infty$. Renormalize $v_j$ such that $v_j(x_j)=1$. Then by classical Harnack inequality \cite{GilbargTrudinger01}, after passing to a subsequence, $v_j\to v_\infty$ in $C^2_{loc}$ away from $\Sing(C)$ for some positive Jacobi field $v_\infty\in C^\infty_{loc}(C\cap \BB_1)$, i.e. $(\Delta_C + |A_C|^2) v_\infty = 0$.
   Again by Harnack inequality Corollary \ref{Cor_Harnack Ineq}, for every $p\in (1, n/(n-2))$ we have, 
   \begin{align*}
    \|v_j\|_{L^p}(\BB_{\eta/4}; \|\Sigma_j\|) \leq C(p, n, \Lambda) v_j(x_j) = C(p, n, \Lambda).
   \end{align*}
   Hence by H\"older inequality, for every small neighborhood $\cN\subset \BB_{\eta/4}$ of $\Sing(C)\cap \BB_{\eta/8}$,  we have \[
    \|v_j\|_{L^1(\cN); \|\Sigma_j\|)} \leq \|v_j\|_{L^p(\cN; \|\Sigma_j\|)} \cdot \|\Sigma_j\|(\cN)^{(p-1)/p} \leq C(p, n, \Lambda)\cdot \|\Sigma_j\|(\cN)^{(p-1)/p}.  \]
   This guarantees that $v_j$ are not $L^1$-concentrating near $\Sing(\Sigma_j)\cap \BB_{\eta/8}$. Hence $v_j$ subconverges to $v_\infty$ also in $L^1(\BB_{\eta/8})$.   By (\ref{Equ_|v_j|_L^1(r_j) < r_j^gamma}) and Lemma \ref{Lem_Growth Jac field, Cone}, we derive,
   \begin{align*}
    \|v_\infty\|_{L^1_*(\BB_{r_\infty}; \|C\|)} & \leq C(n, \Lambda)\inf_{C\cap \BB_{\eta/8}} v_\infty \cdot r_\infty^\gamma \\
    & \leq C(n, \Lambda)\|v_\infty\|_{L^1_*(\BB_{\eta/8}; \|C\|)} \cdot r_\infty^\gamma \\
    & \leq \bar{C}(n, \Lambda)\|v_\infty\|_{L^1_*(\BB_{r_\infty}; \|C\|)}  \cdot (\frac{\eta}{8r_\infty})^{\gamma_n}, 
   \end{align*}
   which then implies \[
     \bar{C}(n, \Lambda)\cdot (\frac{\eta}{8})^{\gamma_n} \geq r_\infty^{\gamma_n - \gamma} \geq r_0^{\gamma_n - \gamma}.   \]
   This will be a contradiction if we take $r_0 < \big(\bar{C}(n, \Lambda)\cdot (\eta/8)^{\gamma_n})^{1/(\gamma_n - \gamma)}$.
  \end{proof}

\section{Minimization with Obstacle} \label{Sec_Minz w obstacle}
  Recall the notion of viscosity mean convex is introduced in Definition \ref{Def_Viscosity mean convex}. Our goal is to prove the following.
  \begin{Thm} \label{Thm_App_Mnmzing with Barrier}
  Let $E$ be a compact viscosity mean convex subset in a Riemannian manifold $(M, g)$ of dimension $n+1$; Let $Z\subset \Int(E)$ be a closed $C^2$ domain (possibly be empty).  Then there exists a closed subset $P$ with the following properties,
   \begin{enumerate} [(i)]
   \item $Z\subset P\subset E$;
   \item The topological boundary $\partial P$ is a stable minimal hypersurface with optimally regularity in the portion $M \setminus Z$ and is $C^{1,1}$ near each point of $\partial P\cap \partial Z$; 
   \item For every Caccioppoli set $Q$ such that $Z\subset Q\subset \subset \Int(E)$, we have $\cP(P)\leq \cP(Q)$;
   \item If $(M, g) \cong (\SSp^{n+1}, g_{\text{round}})$ is isometric to the round sphere, then one can choose so that $P\subset \Int(E)$. 
   \end{enumerate}
   In particular, $P$ is a $C^{1,1}$-optimally regular mean convex subset.
  \end{Thm}
  In the Euclidean space, such existence-with-barrier result was established by \cite[Theorem 3.6]{IlmanenSternbergZiemer98}.  Here we shall use a similar argument.  The only difficulty is that in a general Riemannian manifold, there's usually no isometric translations. So instead of minimizing area, we shall first minimize an $\cA^\epsilon$ functional to find prescribed mean curvature approximations.  Such $\cA^\epsilon$ functional was introduced and carefully studied by \cite{ZhouZhu18}.  
  \begin{proof}
   Suppose WLOG that $injrad(M, g)>1$ and every ball of radius$<1$ in $(M, g)$ is strictly convex. 
   Let $d(x):= dist_g(x, E^c)$ be the distant function defined in $E$; Then there exists $j_0>>1$ such that $Z \subset \{d>1/j_0\}$; Let $\eta\in C^\infty(M; [0, 1])$ be a cut-off function such that $\eta = 0$ in a neighborhood of $Z$ and $\eta = 1$ on $\{d\leq 1/(2j_0)\}$.
    
   Let $\epsilon\in (0, 1)$, consider the the following $\cA^\epsilon$-functional \cite{ZhouZhu18} for a Caccioppoli set $Q$, 
   \begin{align*}
    \cA^\epsilon(Q):= \cP(Q)+ \int_Q \epsilon\cdot \eta(x)\ d\scH^{n+1}(x) .
   \end{align*}
   By first variation \cite{ZhouZhu18}, a stationary Caccioppoli set for $\cA^\epsilon$ has boundary mean curvature equals to $\epsilon\cdot \eta$.\\
   \textbf{Claim 1.} There exists $j_\epsilon\geq 4j_0 + 1/\epsilon$ such that for every $j> j_\epsilon$, there exists a closed Caccioppoli set $Z \subset Q_j\subset \{d> 1/j\}$ which minimize $\cA^\epsilon$ among 
   \begin{align}
     \scT_j := \{Q: Z \subset Q \subset \{d\geq 1/j\} \}.  \label{Append_family scT_j of Cacciop set}    
   \end{align}
   In particular, $\partial Q_j$ is optimally regular and stable in $M\setminus Z$ and $C^{1,1}$ near each point on $\partial Z$.\\
   
   With this Claim 1, choose $Q^\epsilon := Q_{j_\epsilon + 1}$ and send $\epsilon\searrow 0$, by spirit of Lemma \ref{Lem_Cptness of perimeter minzer in E_j} and \cite[Chapter 2]{Lin85}, $Q^\epsilon$ $\mbfF$-subconverges to some closed Caccioppoli set $P = \spt[P]\subset E$, with $\partial P$ stable minimal hypersuraces in $M\setminus Z$ and $C^{1,1}$ near $\partial Z$. This proves (i) and (ii).   Also since $P$ is the $\mbfF$-limit of $\cA^\epsilon$-minimizer $Q^\epsilon$, (iii) also holds for $P$. When $(M, g) \cong (\SSp^{n+1}, g_{\text{round}})$, since $\partial Q^\epsilon$ is $\cA^\epsilon$-stationary and thus has mean curvature $>0$ on $\{d< 1/(2j_0)\}$, by a rotation of $\SSp^{n+1}$ and viscosity mean convexity of $E$, we must have $d\geq 1/(2j_0)$ along $Q^\epsilon$ and then along $P$.  This proves (iv), and finish the proof of Theorem \ref{Thm_App_Mnmzing with Barrier}.\\
   
\noindent \textit{Proof of Claim 1.}  Let $j_\epsilon= j(\epsilon, g)\geq 4j_0 + 1/\epsilon$ to be determined later. 

   For $j>j_\epsilon$, let $\bar{Q}_j\in \scT_j$ be a general $\cA^\epsilon$-minimizer among $\scT_j$, where $\scT_j$ is defined in (\ref{Append_family scT_j of Cacciop set}). For every $x\in \spt[\bar{Q}_j]\cap \{d = 1/j\}$, consider minimize $\cA^\epsilon$ among \[
     \{R\subset M:  R\Delta \bar{Q}_j \subset B_{1/j}(x)\},   \]
   to find a Caccioppoli set $\tilde{Q}_j$. \\
   \textbf{Claim 2.} By taking $j_\epsilon$ large, for every $y\in \spt[\tilde{Q}_j]\cap B_{1/j}(x)$, we have $d(y)> 1/j$. 
   
   With this Claim 2, $\tilde{Q}_j\in \scT_j$ is also an $\cA^\epsilon$-minimizer among $\scT_j$, and whenever $\bar{Q}_j$ is $\cA^\epsilon$-stationary and optimally regular in an open subset $W\subset \{d>1/j\}$, by unique continuation of CMC hypersurfaces \cite{ZhouZhu18}, $\tilde{Q}_j$ must be $\cA^\epsilon$-stationary and optimally regular in $W\cup B_{1/j}(x)$. Hence by covering $\{d = 1/j\}$ with finitely many ball of radius $1/j$ and repeat the replacement process, we can find $Q_j$ satisfying the assertion of Claim 1. \\
   \textit{Proof of Claim 2.} Let $j_\epsilon = j(\epsilon, g)\geq 4j_0+ 1/\epsilon$ TBD.  By definition of $\tilde{Q}_j$ and strict convexity of $B_{1/j}(x)$, we must have\[
     \spt[\tilde{Q}_j]\setminus B_{1/j}(x) \subset \spt[\bar{Q}_j]\setminus B_{1/j}(x) \subset \{d\geq 1/j\}.   \]
   If Claim 2 fails, then there exists $y\in \spt(\partial [\tilde{Q}_j])\cap B_{1/j}(x)$ realizing $\inf \{d(z): z\in \spt(\tilde{Q}_j)\cap B_{1/j}(x)\}$. Let $y'\in \partial E$ and $0< d_0\leq 1/j$ be such that $dist_g(y, y') = d(y) = d_0$.  Since $B_{d_0}(y')\cap \spt(\tilde{Q}_j)=\emptyset$, the tangent cone of $\partial_* \tilde{Q}_j$ at $y$ must be a plane, and hence by $\cA^\epsilon$-minimizing of $\tilde{Q}_j$ in $B_{1/j}(x)$, we know that $y$ is a regular point of $\partial [\tilde{Q}_j]$.  Moreover, by comparing the second fundamental form $\tilde{A}_j$ of $\partial [\tilde{Q}_j]$ and second fundamental form of $\partial B_{d_0}(y')$, together with the fact that the mean curvature of $\partial [\tilde{Q}_j]$ at $y$ is $\epsilon$ by first variation of $\cA^\epsilon$, we know that,
   \begin{align}
    |\tilde{A}_j(y)|\leq C(n, g)\cdot d_0^{-1}.  \label{Append_Bd 2nd fund form |A_j|(y)< d_0^(-1)}
   \end{align}
   
   On the other hand, let $L_y$ be the image of $B_\epsilon(\mathbf{0}) \cap \tilde{\nu}_y^\perp \subset T_y M$ under exponential map under metric $g$ centered at $y$, where $\tilde{\nu}_y$ is the unit normal vector of $\partial [\tilde{Q}_j]$ at $y$ pointing towards $y'$.  We shall work under Fermi coordinates \[
    \Phi: L_y\times (-\epsilon, \epsilon) \to M, \ \ \ (z, t)\mapsto  \exp^g_z(t\cdot \nu^L(z)),  \]
   where $\nu^L$ be the unit normal field of $L_y$ such that $\nu^L(y) = \tilde{\nu}_y$. 
   
   Consider $R := \Phi(\{(z, t+d_0): \Phi(z, t)\in \spt(\tilde{Q}_j)\cap B_{1/j}(x) \})$. Since $dist_g(\Phi(z, t), \Phi(z, t+ d_0)) \leq d_0$, we have $R \subset E$ and $y'= \Phi(y, d_0)\in R$ is a regular point of $\partial R$. Moreover, by (\ref{Append_Bd 2nd fund form |A_j|(y)< d_0^(-1)}) and the following Lemma \ref{Lem_App_mean curv of translated graph}, the mean curvature $H_{\partial R}(y')$ of $\partial R$ at $y'$ satisfies, \[
     H_{\partial R}(y') \geq \tilde{H}_j(y) - \bar{C}(n, g)\cdot d_0 \geq \epsilon - \bar{C}(n, g)\cdot j_\epsilon^{-1},    \] 
   where recall the mean curvature of $\partial [\tilde{Q}_j]$ at $y$ is $\tilde{H}_j(y) = \epsilon$.  Choosing $j_\epsilon > \bar{C}(n, g)\cdot \epsilon^{-1}$, we have $H_{\partial R}(y') > 0$, violates the viscosity mean convexity of $E$ and hence is a contradiction.  This finish the proof of Claim 2.
  \end{proof}
  
  \begin{Lem} \label{Lem_App_mean curv of translated graph}
   Let $\{g^t\}_{t\in (-1,1)}$ be a smooth family of Riemannian metric on $L:= \BB_1^n(\mathbf{0})$ (not necessarily complete); Let $\{A^t:= \frac{1}{2}\frac{d}{dt}(g^t)\}_{t\in (-1,1)}$ be the family of symmetric 2-tensor of derivative of $g^t$.  Assume that $A^0(\mathbf{0}) = 0$.
   
   Consider on the cylinder $L\times (-1, 1)$, the metric $g:= g^t\oplus dt^2$.  Assume that the Riemannian curvature tensor of $g$ satisfies $|Rm_g|\leq 1/2n$ on $L\times (-1, 1)$.
   
   Let $u$ be a $C^2$-function on some small ball $\BB_r(\mathbf{0})\subset L$ taking value in $(-1, 1)$ such that $u(\mathbf{0})=: a$ and $du (\mathbf{0}) = 0$. Let $\Gamma_u := \{(z, u(z)): z\in \BB_r\} \subset L\times (-1, 1)$ be a $C^2$ hypersurface.  Let $\{z^i\}_{1\leq i\leq n}$ be a normal coordinates of $g^0$ defined near $\mathbf{0}$ on $L$.
   Then the mean curvature of $\Gamma_u$ at $(\mathbf{0}, a)$ with respect to the upper-pointed normal field is given by \[
     H_{\Gamma_u}(\mathbf{0}, a) =  (\nabla^2_{g_0} u(\mathbf{0}) - A^a(\mathbf{0}))_{ij}\cdot g_a(\mathbf{0})^{ij},    \]
     
   In particular, the difference between mean curvature of $\Gamma_u$ at $(\mathbf{0}, a)$ and $\Gamma_{u-a}$ at $(\mathbf{0}, 0)$ has estimate \[
     |H_{\Gamma_u}(\mathbf{0}, a) - H_{\Gamma_{u-a}}(\mathbf{0}, 0)| \leq C(n)(a^2|A_{\Gamma_{u-a}}|(\mathbf{0}, 0) + a).    \]
  \end{Lem}
  \begin{proof}
   First note that $\{z^1, ..., z^n, t\}$ is a local coordinates of $L\times (-1, 1)$ with the coordinate vector fields denoted by $\{\partial_1, ..., \partial_n, \partial_t\}$. By \cite{Petersen06RiemGeom}, under this,
   \begin{align*}
    2 A^t (\partial_i, \partial_j) & = \frac{d}{dt} g^t(\partial_i, \partial_j) = \nabla^g_{\partial_t} (g(\partial_i, \partial_j)) = 2g(\nabla^g_{\partial_i} \partial_t, \partial_j); \\
    \frac{d}{dt} A^t(\partial_i, \partial_j) & = \frac{1}{2}g(\nabla^g_{\partial_t}\nabla^g_{\partial_i}\partial_t, \partial_j) + \frac{1}{2}g(\nabla^g_{\partial_t}\nabla^g_{\partial_j}\partial_t, \partial_i) + g(\nabla^g_{\partial_i}\partial_t, \nabla^g_{\partial_j}\partial_t) \\
     & = Rm_g(\partial_t, \partial_i, \partial_t, \partial_j) + A^t(\partial_i, \partial_k)A^t(\partial_j, \partial_l)g_t^{kl},
   \end{align*}
   where $[g_t^{kl}]$ is the inverse of matrix $[g^t_{ij}]$.  In particular, by a bootstrap argument, we have 
   \begin{align}
    |A^t(\partial_i, \partial_j)| \leq \frac{t}{2}, &\  & |g^t _{ij} - \delta_{ij}|\leq \frac{t^2}{2}.  \label{Equ_Error Est on A^t and |g^t -delta|}
   \end{align}
   
   Let $\Phi_u: z\mapsto (z, u(z))$ be a parametrization of $\Gamma_u$.  Then the upward pointed normal field of $\Gamma_u$ at $(\mathbf{0}, a)$ is \[
     \nu^u(\mathbf{0}, a) = \frac{-\partial_i ug_u^{ij}\partial_j + \partial_t}{\sqrt{1+ \partial_i u \partial_j u g_u^{ij}}} = \partial_t;   \]
   The induced metric of $\Gamma_u$ under this parametrization is \[
     g(\partial_i \Phi_u, \partial_j \Phi_u) = g(\partial_i + \partial_i u\cdot \partial_t, \partial_j + \partial_j u \cdot \partial_t);   \]
   And the second fundamental form of $\Gamma_u$ at $(\mathbf{0}, a)$ is \[
     A_{\Gamma_u}(\partial_i \Phi_u, \partial_j \Phi_u)|_{(\mathbf{0}, a)} = \langle \nabla^g_{\partial_i + \partial_i u\cdot\partial_t} (\partial_j + \partial_ju\cdot \partial_t), \nu^u \rangle|_{(\mathbf{0}, a)} = - A^a(\mathbf{0}) + \nabla^2_{g^0}u(\mathbf{0})(\partial_i, \partial_j).   \]
   Hence the mean curvature of $\Gamma_u$ at $(\mathbf{0}, a)$ is \[
     H_{\Gamma_u}(\mathbf{0}, a) = (\nabla_{g^0}^2 u (\mathbf{0})- A^a(\mathbf{0}))_{ij}\cdot g_a(\mathbf{0})^{ij};   \]
   The estimate on the difference of mean curvature follows from the estimate (\ref{Equ_Error Est on A^t and |g^t -delta|}).
  \end{proof}
  
  \begin{Cor} \label{Cor_C^1,1 optimal reg mean convex approx}
   Let $\cE \subset (\SSp^{n+1}, g_{\text{round}})$ be a closed viscosity mean convex subset with connected interior.  Then there exists an increasing sequence of closed connected $C^{1,1}$ optimally regular mean convex subsets $\{\cE_j \subset \Int(\cE)\}_{j \geq 1}$ such that when $j\to \infty$, $\cE_j \to \cE$, $\partial \cE_j \to \partial \cE$ both in Hausdorff distant sense.  
   Furthermore, if the topological boundary $\partial \cE$ is $n$-rectifiable (In particular $\cE$ is a Caccioppoli set),  then we can choose so that \[
     \lim_{j\to \infty} \cP(\cE_j) = \cP(\cE).   \]
  \end{Cor}
  \begin{proof}
   Let $d_{\cE}:= dist_{\SSp^{n+1}}(\cdot, \cE^c)$ be the distant function to the boundary defined in $\cE$.  $\tilde{d}_{\cE}$ be a smooth function on $\Int(\cE)$ such that $d_{\cE} \leq \tilde{d}_{\cE} \leq d_{\cE}$. Let $\epsilon_j\searrow 0$ be a sequence of regular value of $\tilde{d}_{\cE}$,  and $Z_j:= \{\tilde{d}_{\cE} \geq \epsilon_j\} \subset \Int(\cE)$ are compact smooth domains. 
   
   By Theorem \ref{Thm_App_Mnmzing with Barrier}, there exists $C^{1,1}$-optimally regular closed mean convex subset $Z_j \subset P_j \subset \Int(\cE)$. Clearly by definition, $P_j\to \cE$ and $\partial P_j \to \partial \cE$ both in Hausdorff distant sense.   Also, let $j_0>>1$ such that $Z_{j_0}\neq \emptyset$ and fix $\omega_0\in Z_{j_0}$;  Take $\cE_j$ to be the connect component of $P_j$ containing $\omega_0$, we know that since $\Int(\cE)$ is connected, $\cE_j\to \cE$ , $\partial \cE_j \to \partial \cE$ also in Hausdorff distant sense.
   
   If further $\partial \cE$ is $n$-rectifiable, by \cite[Theorem 3.2.39]{Federer69} and co-area formula, there exists $R_j := \{d_{\cE} \geq \delta_j\} \subset \Int(\cE)$ such that $\lim_{j\to \infty}\cP(R_j) = \cP(\cE)$.  Then by Theorem \ref{Thm_App_Mnmzing with Barrier} (iii) and upper semi-continuity of perimeter under convergence, we have $\lim_{j\to \infty} \cP(\cE_j) = \cP(\cE)$.
  \end{proof}

\bibliographystyle{alpha}
\bibliography{GMT}
\end{document}